\documentclass[11pt, leqno]{amsart}

\ifdefined\NewCommandCopy
  \NewCommandCopy\latexbibitem\bibitem
\else
  \let\latexbibitem\bibitem
  \usepackage{xparse}
\fi

\ExplSyntaxOn
\RenewDocumentCommand{\bibitem}{O{}m}
 {
  \lullaby_bibitem:en { \lullaby_bibitem_check:nn { #1 } { \exp_not:n {#1}} } {#2}
 }

\cs_new_protected:Nn \lullaby_bibitem:nn { \latexbibitem[#1]{#2} }
\cs_generate_variant:Nn \lullaby_bibitem:nn { e }
\cs_new:Nn \lullaby_bibitem_check:nn
 {
  \str_case:nnF {#1}
   {
    {Koe08}{Koe1908}
    {Koe20}{Koe1920}
   }
   {#2}
 }
\ExplSyntaxOff

\usepackage[leqno]{amsmath}
\usepackage{amssymb}
\usepackage{enumerate}

\usepackage{subfigure}
\usepackage{mathrsfs}
\usepackage{graphicx}
\usepackage{bm}
\usepackage[all]{xy}
   \usepackage[margin=1in]{geometry}

\usepackage[usenames,dvipsnames]{xcolor}
\usepackage[
colorlinks=true,linkcolor=NavyBlue,urlcolor=RoyalBlue,citecolor=PineGreen,bookmarks=true,bookmarksopen=true,bookmarksopenlevel=2,unicode=true,linktocpage]{hyperref}

\usepackage{microtype}
\usepackage{centernot}
\usepackage{stmaryrd}
\usepackage{comment}

\numberwithin{equation}{section}

\newtheorem{theorem}{Theorem}[section]

\newtheorem{lemma}[theorem]{Lemma}
\newtheorem{proposition}[theorem]{Proposition}

\theoremstyle{remark}
\newtheorem{remark}[theorem]{Remark}

\newcommand{\C}{\mathbf{C}}
\newcommand{\D}{\mathbf{D}}
\newcommand{\E}{\mathbf{E}}

\newcommand{\h}{\mathbf{H}}
\newcommand{\N}{\mathbf{N}}
\newcommand{\Z}{\mathbf{Z}}
\newcommand{\p}{\mathbf{P}}
\newcommand{\Q}{\mathbf{Q}}
\newcommand{\R}{\mathbf{R}}

\newcommand{\fh}{\mathfrak {h}}

\newcommand{\fp}{\mathfrak {p}}

\newcommand{\CL}{\mathcal {L}}

\newcommand{\CS}{\mathcal {S}}

\newcommand{\CZ}{\mathcal {Z}}

\newcommand{\cZ}{\mathcal {Z}}

\newcommand{\SLE}{{\rm SLE}}
\newcommand{\CLE}{{\rm CLE}}

\newcommand{\dist}{\mathrm{dist}}

\newcommand{\diam}{\mathrm{diam}}
\newcommand{\Area}{\mathrm{Area}}

\newcommand{\im}{\mathrm{Im}}

\newcommand{\interior}{\mathrm{int}}

\newcommand{\confrad}{{\rm CR}}

\newcommand{\one}{{\bf 1}}

\newcommand{\wt}{\widetilde}
\newcommand{\wh}{\widehat}

\newcommand{\giv}{\,|\,}

\newcommand{\sleu}{{\SLE_{\kk'}(\kk'-6)}}
\newcommand{\slekk}{{\SLE_{\kk}(\kk-6)}}

\newcommand{\slekp}{\SLE_{\kk'}}
\newcommand{\clekp}{\CLE_{\kk'}}

\newcommand{\ee}{\varepsilon}
\newcommand{\dd}{\delta}
\newcommand{\aal}{\alpha}
\newcommand{\bb}{\beta}
\renewcommand{\gg}{\gamma}
\newcommand{\GG}{\Gamma}
\newcommand{\la}{\lambda}
\newcommand{\ff}{\varphi}

\newcommand{\tht}{\theta}

\newcommand{\kk}{\kappa}
\renewcommand{\ss}{\sigma}

\newcommand{\Uu}{\Upsilon}

\def\n#1{\left\lvert#1\right\rvert}
\def\nn#1{\left\lVert#1\right\rVert}
\newcommand{\nv}{^{-1}}

\newcommand{\del}{\partial}
\newcommand{\sm}{\setminus}

\newcommand{\ov}{\overline}

\newcommand{\cci}{C_c^\infty}
\newcommand{\cciz}{C_{c,0}^\infty}

\let\originalleft\left
\let\originalright\right
\renewcommand{\left}{\mathopen{}\mathclose\bgroup\originalleft}
\renewcommand{\right}{\aftergroup\egroup\originalright}

\begin{document}

\title{The square summability of the CLE complementary component diameters}

\author{Cillian Doherty and Jason Miller}

\begin{abstract}
We show that the sum of the squares of the diameters of the complementary connected components of the $\CLE_\kk$ carpet/gasket is almost surely finite for $\kk \in (8/3, 4) \cup (4, 8)$.
This is a prerequisite for the application of a result of Ntalampekos which allows the $\CLE_\kk$ carpet/gasket to be uniformized to a round Sierpi\'nski packing, in analogy with the classical Koebe uniformization theorem for finitely connected domains.  Our result is new in the case that $\kk \in (4,8)$ and we provide a new proof for $\kk \in (8/3, 4)$.  In both cases we use the link between CLE and space-filling SLE.
The square-summability of diameters has been proved for $\kk \in (8/3, 4]$ in unpublished work by Rohde and Werness using a different method.
Our work completes the proof that this property holds for all $\kk$ for which $\CLE_\kk$ is defined. 
\end{abstract}

\date{\today}
\maketitle

\setcounter{tocdepth}{1}

\parindent 0 pt
\setlength{\parskip}{0.2cm plus1mm minus1mm}

\section{Introduction}
Conformal loop ensembles (CLE) are a one-parameter family of random countable collections of loops in a simply connected domain $D \subsetneq \C$, first introduced and studied in \cite{s2009cle} and \cite{sw2012markovian}.
They are the unique random family of loops whose law is conformally invariant and satisfies a certain restriction property (see Section~\ref{sec:cle}).
CLE can be viewed as a loop version of Schramm--Loewner evolution (SLE), a family of random curves with conformal invariance and Markov properties.
CLE and SLE are conjectured, and in some instances proved, to be the scaling limits of interfaces in a wide range of statistical physics models (see e.g.\ \cite{s2001cardy, camia_newman_full_scaling_limit, benoist_hongler_2019_scaling_limit_ising, ksl2024}).
CLE is parameterised by a value $\kk \in (8/3, 8)$. The three regimes $\kk \in (8/3, 4), \kk = 4$ and $\kk \in (4,8)$ exhibit qualitatively different behaviour and often need to be considered separately when proving general results.
For $\kk \in (8/3,4]$, the loops in a $\CLE_\kk$ are a.s.\ simple and do not intersect each other, whereas for $\kk \in (4,8)$ the loops a.s.\ intersect themselves and each other (although they are not self-crossing and do not cross each other). 
Throughout this paper, we will consider \textit{non-nested CLE}, meaning that a.s.\ no loop is contained inside another.

Throughout this paper, $\GG = (\ell_k)_{k \in \N}$ will denote a collection of loops distributed according to the law of a (non-nested) $\CLE_\kk$. Any fixed point $z \in D$ is a.s.\ contained inside a single loop of $\GG$, but there will exist exceptional points which are not contained in any loop.
We refer to the set of points in $\D$ not contained in any loop as the $\CLE_\kk$ \textit{carpet} for $\kk \in (8/3, 4]$ and the $\CLE_{\kk}$ \textit{gasket} for $\kk \in (4,8)$ (in analogy with the Sierpi\'nski carpet and gasket) and denote it by $\Uu$.
For all $\kk \in (8/3, 8)$ the set $\Uu$ a.s.\ has empty interior and if we let (in a slight abuse of notation) $\interior(\ell_k)$ denote the points surrounded by the loop $\ell_k$, it can be written as $\Uu = \D \sm \bigcup_{k \in \N} \,\interior(\ell_k)$.
When $\kk \in (8/3, 4]$, the interior of each loop is a.s.\ a connected set (and a Jordan domain), so the connected components of $\D\sm\Uu$ coincide with the interiors of the loops.
When $\kk \in (4,8)$, however, the interior of each loop is a.s.\ not connected, so there are many connected components corresponding to each CLE loop.
We denote by $(s_k)_{k \in \N}$ the collection of connected components of $\D\sm\Uu$, emphasising that $s_k$ can be identified with $\interior(\ell_k)$ when $\kk \in (8/3, 4]$, but not when $\kk \in (4,8)$. 
Our main result is the following.

\begin{theorem}\label{thm:diam}
Let $\kk \in (\tfrac83, 8) \sm \{4\}$, let $\GG$ be a $\CLE_\kk$
in $\D$ and let $\Uu$ denote its carpet/gasket. Enumerate the connected components of $\D \sm \Uu$ as $(s_k)_{k \in \N}$. Then, almost surely,
\begin{equation}\label{eq:diam_sum}
    \sum_{k = 1}^\infty \diam(s_k)^2 < \infty.
\end{equation}
\end{theorem}

The above result has previously been proved in unpublished work by Rohde and Werness for the case $\kk \in (8/3, 4]$ \cite{rohde_werness2015}.
Details of this result will be included in \cite{rohde_upcoming}. 
They take a different approach to us to prove \eqref{eq:diam_sum}, and make use of the Brownian loop soup construction of $\CLE_\kk$ described in \cite{sw2012markovian}, which is only valid when $\kk \in (8/3, 4]$.
We instead use the link between CLE and space-filling SLE which applies for all $\kk \in (8/3, 8)$ except for $\kk = 4$. This leads to the two methods giving different ranges of $\kk$ where \eqref{eq:diam_sum} holds. 
Combining the two results, however, we can conclude that Theorem~\ref{thm:diam} holds for all values of $\kk \in (8/3, 8)$.

The significance of Theorem~\ref{thm:diam} lies in the context of the problem of \textit{uniformizing} the $\CLE_\kk$ carpet and gasket, and the uniformization of \textit{Sierpi\'nski carpets} more generally. 
A Sierpi\'nski carpet\footnote{What we have called a Sierpi\'nski carpet is sometimes referred to as a \textit{topological Sierpi\'nski carpet}, with the term Sierpi\'nski carpet being reserved for the standard Sierpi\'nski carpet. We have decided for simplicity to keep with the terminology used in \cite{ntalampekos_conformal_uniformization}.} is a closed subset of the Riemann sphere with empty interior which can be written as $\wh\C \sm \bigcup_{k \in \N} s_k$, where the sets $s_k$ are Jordan domains whose closures are pairwise disjoint, and where $\diam(s_k) \to 0$ in the spherical metric as $k \to \infty$. 
A classical result \cite{whyburn} states that all Sierpi\'nski carpets are homeomorphic to the standard Sierpi\'nski carpet and hence to each other.
For all $\kk \in (8/3, 8)$, $\Uu$ is a.s.\ a Sierpi\'nski carpet (when viewed as a subset of $\wh\C$). This is easy to see  when $\kk \in (8/3, 4]$, since $\Uu$ is exactly the set retained after removing countably many non-intersecting Jordan domains from $\D$, but in the $\kk \in (4,8)$ regime it is not as obvious and is proved here in Proposition~\ref{prop:gasket_is_carpet}.

The study of the uniformization of planar domains has a long history, with Koebe conjecturing that every domain in the Riemann sphere can be conformally mapped to a \textit{circle domain}, a domain in the sphere whose complement consists of a collection of disjoint closed disks and points \cite{koebe1908}.
This conjecture was has since been proved for finitely connected domains \cite{koebe1920}, countably connected domains \cite{he_schramm1993} and for the class of so-called \textit{cofat} domains which may not be countably connected \cite{schramm1995}. The general conjecture remains open, however.
The uniformization of \textit{Sierpi\'nski carpets} has more recently been studied in \cite{bonk2011, ntalampekos2020potential_theory, ntalampekos_conformal_uniformization}.
To uniformize a Sierpi\'nski carpet $S$ one wants to map $S$ to a \textit{round Sierpi\'nski carpet}, which is a Sierpi\'nski carpet whose complementary components are disks. In the case of planar domains one requires the uniformizing map to be conformal on the domain.
In the case of Sierpi\'nski carpets, one needs also to define suitable conditions on the map, although exactly what conditions one requires varies depending on the generality of the carpets considered. 
We refer to \cite{ntalampekos_conformal_uniformization} for further discussion.

In \cite{bonk2011}, a uniformization result was proved for Sierpi\'nski carpets which satisfy certain geometric conditions. There, the uniformizing map can be taken to be quasisymmetric and is shown to be unique up to M\"obius transformations.
The $\CLE_\kk$ carpet/gasket a.s.\ does not satisfy these geometric conditions however so these results cannot be applied in this case.
In \cite{ntalampekos_conformal_uniformization}, Ntalampekos introduced the concept of a \textit{Sierpi\'nski packing} which
is
a set $Y \subset \wh\C$ that can be written as $Y = \wh\C \sm \bigcup_{k \in \N} q_k$ where the $q_k$ are pairwise disjoint closed Jordan domains\footnote{In \cite{ntalampekos_conformal_uniformization}, the sets $q_k$ are allowed to be closed connected sets whose complements are connected, but to streamline our presentation we are more restrictive here.} and where $\diam (q_k) \to 0$ as $k \to \infty$.
Ntalampekos showed that if $Y$ is a Sierpi\'nski packing where $\sum_k \diam(q_k)^2 < \infty$, then $Y$ can be 
uniformized in the following sense. There 
exists a \textit{round Sierpi\'nski packing} $X = \wh\C\sm\bigcup_{k \in \N}p_k$, where each $p_k$ is now also required to be a (non-degenerate) disk, and a continuous map $H \colon \wh\C \to \wh\C$ which satisfies $H\nv(\interior (q_k)) = \interior (p_k)$ for all $k \in \N$, and satisfies also the other requirements of \cite[Theorem~1.1]{ntalampekos_conformal_uniformization}. 
If $C = \wh\C \sm \bigcup_{k \in \N} s_k$ is a Sierpi\'nski carpet, there is a corresponding Sierpi\'nski packing $Y = \wh\C\sm\bigcup_{k\in\N}\ov{s_k}$, and we have that $\ov{Y} = C$.
It follows that any Sierpi\'nski carpet satisfying $\sum_k \diam(s_k)^2 < \infty$ can be uniformized in this way.
In particular, in contrast to \cite{bonk2011} there is no geometric requirement which means that Theorem~\ref{thm:diam} (along with Proposition~\ref{prop:gasket_is_carpet} and the corresponding result for $\kk = 4$ from \cite{rohde_werness2015}) shows immediately that this result can be applied\footnote{Using \cite{rohde_werness2015}, this result is stated in \cite{ntalampekos_conformal_uniformization} as Corollary~1.4 with the assumption that $\kk \in (8/3,4]$. Due to Theorem~\ref{thm:diam}, this corollary now holds for the full range $\kk \in (8/3, 8)$.}  to the $\CLE_\kk$ carpet/gasket for all $\kk \in (8/3, 8)$.

If $X$ is the round Sierpi\'nski packing one obtains after applying \cite[Theorem~1.1]{ntalampekos_conformal_uniformization}, it is not known whether $\ov{X}$ is a Sierpi\'nski carpet, since it may not have empty interior. Also, unlike in \cite{bonk2011}, uniqueness is not proved for either the uniformizing map or the resulting round packing.
If one could prove uniqueness of the packing (up to M\"obius transformations) one would obtain a natural law on random round Sierpi\'nski packings arising from CLE.
It would be interesting then to ask if there is an axiomatic characterisation of this law (like there is for CLE itself) and also how it depends on the parameter $\kk$.
We refer also to \cite[Questions~1.5 and~1.6]{ntalampekos_conformal_uniformization} for a similar discussion of such implications.
CLE carpet uniformization for $\kk \in (8/3, 4]$ also featured in \cite{rohde_werness2015} and will be studied in \cite{rohde_upcoming}.

\subsection*{Notation}
In the remainder of this paper, $\kk$ will be a value in $(8/3, 8)$ which is not equal to $4$. We define $\kk' = 16/\kk$ if $\kk \in (8/3, 4)$ and $\kk' = \kk$ if $\kk \in (4,8)$, so that $\kk'$ is always in $(4,8)$. 
While usually (as is often the convention) we use the symbol $\kk$ when $\kk \in (8/3,4)$ and $\kk'$ when it is in $(4,8)$, there are situations when we want to discuss both cases simultaneously. 
This leads us to sometimes using $\kk$ for the full range and defining $\kk'$ in the slightly strange way above.

\subsection*{Outline} 
Next we give an outline of both the paper and the proof of Theorem~\ref{thm:diam}. In Section~\ref{sec:prelims} we define Schramm--Loewner evolution (SLE), the Gaussian free field (GFF) and the relationship between them. 
In Section~\ref{sec:cle}, we recall (using the results of \cite{s2009cle, ms2016imag1, msw2017cleperc}) that for\footnote{This is also true for $\kk = 4$ but we do not include the construction here.} $\kk \in (8/3, 8) \sm \{4\}$, an instance $\GG$ of a $\CLE_\kk$ can be coupled with a particular GFF $h$ on $\D$ in such a way that $\GG$ is measurable with respect to $h$.
A space-filling $\SLE_{\kk'}(\kk'-6)$ loop $\eta'$ on $\D$ can also be coupled with $h$ so that it too is a measurable function of the field (see Section~\ref{sec:sfsle}).
Lemmas~\ref{lem:sfsle_pocket_interaction_48} and~\ref{lem:sfsle_pocket_interaction_834} show that under this coupling, $\eta'$ fills in each pocket of $\D\sm\Uu$ \textit{in one go}; that is, once $\eta'$ enters this pocket, it hits every point in the pocket before leaving its closure. 
We remark that the interaction between the space-filling $\SLE_{\kk'}(\kk'-6)$ curve $\eta'$ and the loops of $\GG$ is different depending on whether $\GG$ is a $\CLE_\kk$ (for $\kk \in (8/3,4)$) or a $\CLE_{\kk'}$ (when $\kk = \kk' \in (4,8)$), but that in both cases, it fills in each pocket of $\D\sm\Uu$ in one go.
In Section~\ref{sec:area_filling}, we extend a result of \cite{ghm2020almost} to prove that with large probability $\eta'$ does not travel a large distance without filling in a ball of a comparable area.
Finally, in Section~\ref{sec:proof}, we combine the previous results to show that, since $\eta'$ fills in pockets of $D\sm\Uu$ in one go, and $\eta'$ satisfies the area-filling property of Section~\ref{sec:area_filling}, the pockets of $\D\sm\Uu$ will with large probability not have large diameter but small area which allows us to prove Theorem~\ref{thm:diam}.

\subsection*{Acknowledgements}  J.M.\ was supported by the ERC consolidator grant ARPF (Horizon Europe UKRI G120614).
C.D.\ was supported by EPSRC grant EP/W524633/1 and a studentship from
Peterhouse, Cambridge.

\section{Preliminaries}\label{sec:prelims}
\subsection{Schramm--Loewner evolution}

\subsubsection{Chordal SLE}
For a more detailed description, we refer to \cite{lawler2008conformally}. For $z \in \h$ and a continuous function $W\colon [0,\infty) \to \R$ called the \textit{driving function}, the Loewner differential equation is given by
\begin{equation}\label{eq:chordal_loewner}
\frac{\del g_t(z)}{\del t} = \frac{2}{g_t(z) - W_t}, \qquad g_0(z) = z.
\end{equation}
For a fixed $z \in \h$, $g_t(z)$ has a unique solution up until $\tau_z = \sup\{t \geq 0 \colon \im\, g_t(z) > 0\}$. For a fixed $t > 0$, let $K_t = \{z \in \h \colon \tau_z \leq t\}$. We call $K_t$ the Loewner hull at time $t$, and view $(K_t)_{t \geq 0}$ as an increasing collection of hulls in $\h$. 
To define Schramm--Loewner evolution (SLE), we let $W$ be the random function $W_t = \sqrt{\kk} B_t$, where $B$ is a one-dimensional Brownian motion started from $0$ and $\kk \in (0,\infty)$. In this case, it has been shown in \cite{rs2005basic, lsw2004lerw, am2022sle8} that there a.s.\ exists a (unique) random continuous curve $\eta \colon [0,\infty) \to \ov\h$ which generates the hulls $K_t$ in the sense that $K_t$ is the complement in $\h$ of the unbounded component of $\h\sm \eta([0,t])$.
We say that $\eta$ is distributed as an $\SLE_\kk$. Almost surely, for $\kk \in (0, 4]$, the curve is simple, for $\kk \in (4,8)$ it intersects itself but is not space-filling, and for $\kk \geq 8$ it is a.s.\ space-filling.
For all values of $\kk$, the curve $\eta$ does not cross itself.

We can generalise SLE by allowing the driving function $W$ to have a different law. The $\SLE_{\kk}(\rho)$ process with \textit {force point} at $a \in \R$ is defined by considering the SDE
\begin{equation}\label{eq:chordal_force_point}
    dW_t = \sqrt{\kk}dB_t + \frac{\rho}{W_t - O_t}dt, \qquad dO_t = \frac{2}{O_t - W_t}dt,
\end{equation}
where $W_0 = 0$ and $O_0 = a$. When $\rho > -2$, there exists a unique solution $(W, O)$ to this SDE if we further require that $W - O$ does not change sign and is instantaneously reflecting at $0$ (in the sense that the set of times that $W_t = O_t$ has Lebesgue measure $0$), and that $(W, O)$ is adapted to the filtration generated by $B$. If $a = 0$ (which will be the case in this paper), we specify either $a = 0^+$ or $a = 0^-$. If $a = 0^+$ (resp.\ $a = 0^-$) we require that $O_t - W_t$ is always non-negative (resp.\ non-positive) and we view the force point as lying an infinitesimal distance to the right (resp.\ left) of $0$. We refer to \cite{ms2016imag1} for further details.
Letting $W$ be the driving function in \eqref{eq:chordal_loewner}, as in the case of standard $\SLE_\kk$ we again obtain a random continuous curve $\eta$ which generates the random hulls $(K_t)$ (see e.g.\ \cite{ms2016imag1}). We say that $\eta$ has the law of an $\SLE_\kk(\rho)$ with force point at $a$.
We can extend the definition of $\SLE_\kk(\rho)$ to other simply connected domains $D \subsetneq \C$ via conformal mapping (see e.g.\ \cite{ms2016imag1}).
One can also consider $\SLE_{\kk}(\rho_L;\rho_R)$ with two force points (or more). We refer to \cite{ms2016imag1} for its definition and properties in this case.

\subsubsection{Radial SLE}\label{subsubsec:radial_sle}
We will define radial $\SLE_\kk(\rho)$ in the case that $\rho > -2$ using the radial Loewner equation,
\begin{equation}\label{eq:radial_loewner}
    \frac{\del g_t(z)}{\del t} = g_t(z) \frac{W_t + g_t(z)}{W_t - g_t(z)},\qquad g_0(z) = z.
\end{equation}
Here, $W \colon [0,\infty) \to \del \D$ is a continuous function again called the driving function. Similar to the chordal case, for $z \in \D$, $g_t(z)$ has a unique solution up until the time $\tau_z = \sup\{t > 0 \colon g_t(z) \in \D\}$, and we define the hull $K_t = \{z \in \D \colon \tau_z \leq t\}$. 
When $W_t = e^{i\sqrt{\kk}B_t}$ where $B$ is a one-dimensional Brownian motion such that $W_0 = w$, and $\kk \in (0,\infty)$, the hulls $(K_t)$ are a.s.\ generated by a continuous random curve $\eta$ in the sense that $K_t$ is the complement (in $\D$) of the connected component of $\D \sm \eta([0,t])$ containing $0$. The curve $\eta$ starts at $w$ and tends to $0$ as $t \to \infty$ a.s. 
We say that $\eta$ is a radial $\SLE_\kk$ from $w$ to $0$. We can define radial $\SLE_\kk$ in any other simply connected domain $\D \subsetneq \C$ from any prime end of $D$ to a point $z \in D$ by conformal mapping.

As in the chordal case, we can introduce a force point and define radial $\SLE_\kk(\rho)$. Suppose that the pair $(W,O)$ satisfies
\begin{equation}
    dW_t = -\frac{\kk}{2}W_t \,dt + i\sqrt{\kk}\,W_t \,dB_t - \frac{\rho}{2} W_t \frac{W_t + O_t}{W_t - O_t}dt, \qquad dO_t = O_t \frac{W_t + O_t}{W_t - O_t}dt,
\end{equation}
where $W_o = w$ and $O_0 = a$. This equation has a unique solution up until the first time that $W_t = O_t$. We will (only) need to consider the case that $W_0 = O_0$, however, so we will need to find a way to define a solution $(W,O)$ which is valid for all time. To this end, following \cite{msw2014gasket}, for $\rho > -\kk/2 - 2$ we define the continuous process $\tht \colon [0,\infty) \to [0, 2\pi]$ which satisfies the SDE
\begin{equation}
    d\tht_t = \sqrt{\kk}dB_t + \frac{\rho + 2}{2}\cot(\tht_t/2)dt
\end{equation}
on intervals when $\tht \in (0,2\pi)$, and which is instantaneously reflecting at its endpoints in the sense that the set of times $\{t \colon \tht_t \in \{0,2\pi\}\}$ has Lebesgue measure $0$. This process is pathwise unique and its law is determined by $\tht_0$.

We then define
\begin{equation}
    \arg W_t = \arg w + \sqrt{\kk} B_t + \frac{\rho}{2}\int_0^t \cot(\tht_s/2)ds,
\end{equation}
which is a.s.\ finite (see \cite{msw2014gasket}). We define radial $\SLE_\kk(\rho)$ starting from $w$ with force point at $o$ to be the solution to \eqref{eq:radial_loewner} with driving function $W$, where $\tht_0$ is determined by $w = o e^{i\tht_0}$. 
If $w = o$, we again need to specify whether the force point lies infinitesimally to the right or left of $w$. If $w = oe^{i0^+}$ (resp.\ $w = oe^{i0^-}$) we set $\tht_0 = 0$ (resp.\ $\tht_0 = 2\pi$) and view the force point as lying infinitesimally to the left (resp.\ right) of $w$. 
For $\kk' \in (4,8)$ and $\rho = \kk' - 6$, the $\SLE_{\kk'}(\kk' - 6)$ process is a.s.\ generated by a continuous curve $\eta'$ which starts at $w$ and targets $0$. We refer to \cite{msw2014gasket} for further details.
We can define $\SLE_{\kk'}(\kk' - 6)$ from $x \in \del D$ to $z \in D$ in any simply connected domain $D \subsetneq \C$ by conformal mapping.
As in the chordal case, the $\SLE_{\kk'}(\kk' - 6)$ processes are target invariant in the sense that if $\eta'_u$ and $\eta'_v$ are $\SLE_{\kk'}(\kk' - 6)$ curves from a point $w \in \del \D$ to points $u,v \in \del \D$, then $\eta'_u$ and $\eta'_v$ can be coupled in such a way that they are equal (up to reparametrisation) until the time that $\eta'_u$ separates $u$ from $v$, and evolve independently thereafter \cite{s2009cle}.

\subsubsection{Generalised $\SLE_\kk(\rho)$ for $\rho < -2$}\label{subsubsec:gen_sle}
When $\rho \leq -2$, it is more difficult to define $\SLE_\kk(\rho)$ processes. We will be interested in the range $\rho \in (-2 - \kk/2, -2)$, for which a family of processes denoted by $\SLE_\kk^\bb(\rho)$, for $\bb \in [-1,1]$, were introduced in \cite{s2009cle}. In particular, we will have $\kk \in (8/3, 4)$ and $\rho = \kk - 6$.
In this paper, we will need only the $\bb = 1$ case, so we will restrict our attention to this, and denote $\SLE_\kk^\bb(\rho)$ simply by $\SLE_\kk(\rho)$. We will first define the chordal $\SLE_\kk(\kk-6)$ process on $\h$ from $0$ to $\infty$. We will partially follow the description of these processes given in \cite{werner_wu}.

Suppose that we have a pair of processes $(W, O)$ satisfying \eqref{eq:chordal_force_point}. Define $X_t = (W_t - O_t)/\sqrt{\kk}$. By \eqref{eq:chordal_force_point}, $X$ satisfies
\begin{equation}\label{eq:bessel}
    dX_t = dB_t + \frac{\kk - 4}{\kk X_t}dt,
\end{equation}
meaning that $X$ is a Bessel process of dimension $\dd = 3 - 8/\kk \in (0,1)$ (see e.g.\ \cite[Chapter XI]{revuz2013continuous}). There exists a solution to \eqref{eq:chordal_force_point} up until the first time that $X_t = O_t$, but we will need to consider the case that $W_0 = O_0 = 0$, so this alone does not suffice.
In order to construct a solution which is valid for all time (which satisfies \eqref{eq:chordal_force_point} when $W_t \neq O_t$) we will take a different approach.

We first define $Z$ to be the square Bessel process of dimension $\dd \in (0,1)$, starting from $z_0 \geq 0$, which is the unique (non-negative) strong solution to the SDE
\begin{equation}
    dZ_t = 2\sqrt{Z_t}dB_t + \dd dt.
\end{equation}
This solution exists for all time (see \cite[Chapter XI]{revuz2013continuous} for more information on Bessel and square Bessel processes).
We then define $X_t = \sqrt{Z_t}$ which we call the Bessel process of dimension $\dd$ started from $\sqrt{z_0}$. Then, $X_t$ satisfies \eqref{eq:bessel} on time intervals where $X_t > 0$, but does not satisfy the integrated form of the equation for all time, since the integral $\int X_t\nv dt$ in general will not converge when $\dd \leq 1$.

However, as explained in \cite{s2009cle} (see also \cite{werner_wu}) we can define the principal value of the integral $\int X_t\nv dt$ and obtain a process $I_t$ which is a.s.\ continuous, satisfies $dI_t = (1/X_t)dt$ on intervals where $X_t > 0$, and which is adapted to the filtration of $B$. 
We then define $W_t = \sqrt{\kk} X_t + 2\sqrt{\kk}I_t$ and $O_t = 2\sqrt{\kk}I_t$. The pair $(W,O)$ is then seen to satisfy \eqref{eq:chordal_force_point} whenever $W_t - O_t \neq 0$ (note that $W_t - O_t$ is non-negative) and which is instantaneously reflecting (in the sense that the set of times where $W_t = O_t$ has Lebesgue measure $0$ a.s.).
We then define $(K_t)$ to be the Loewner hulls obtained from \eqref{eq:chordal_loewner} with driving function equal to $W$. It was shown in \cite{msw2017cleperc} that there a.s.\ exists a curve $\eta$ which generates the hulls $K_t$. 
Furthermore, in \cite{msw2017cleperc} it was shown that $\eta$ consists of a trunk, $\eta'$, which has the law of an $\SLE_{\kk'}(\kk'-6)$ and loops hanging off the right side of this trunk, each of which corresponds to an excursion of the process $W_t - O_t$ away from $0$.
We call these loops \textit{excursion loops} of $\eta$ to differentiate them from loops $\eta$ may make corresponding to the trunk $\eta'$ itself tracing a loop.
The trunk is exactly the set of points traced by $\eta$ at times when $O_t = W_t$.

As in the case of standard $\SLE_\kk(\rho)$, we can extend the definitions to arbitrary simply connected domains $D \subsetneq \C$ and different start and endpoints using conformal mapping. 
It is explained in \cite{msw2017cleperc} that $\SLE_\kk(\kk-6)$ processes continue to have a target invariance property when $\kk \in (8/3, 4)$. Precisely, if $\eta_u$ and $\eta_v$ are two $\SLE_\kk(\kk-6)$ curves in $D$ started from $x$ and targeted at $u$ and $v$ in $\del D$ respectively, then $\eta_u$ and $\eta_v$ can be coupled so that they agree until they separate $u$ and $v$, and evolve independently afterwards.

To define radial $\SLE_\kk(\kk-6)$ when $\kk \in (8/3, 4)$ we do not use the radial Loewner equation, but instead build the curve iteratively using chordal $\SLE_\kk(\kk-6)$ and target invariance  (following \cite{msw2017cleperc} and \cite{werner_wu}).
Let $D \subsetneq \C$ be a simply connected domain with $x \in \del D$ a prime end and $z \in D$. To define the radial $\SLE_\kk(\kk-6)$ from $x$ to $z$ we first choose a point $y_1 \in \del D$ and sample an $\SLE_\kk(\kk-6)$ from $x$ to $y_1$ up until the first time, $\tau_1$, that $y_1$ is separated from $z$.
Let $D_2$ be the connected component of $D \sm \eta_1([0,\tau_1])$ containing $z$. Choose $y_2 \in \del D_2$ (in some measurable way) and let $\eta_2$ be an $\SLE_\kk(\kk-6)$ from $\eta_1(\tau_1)$ to $y_2$, up until $\eta_2$ separates $y_2$ from $z$. 
We continue in this way, defining curves $(\eta_n)$ for $n \in \N$. We let $\eta$ be the curve obtained by joining these curves in order and call $\eta$ the radial $\SLE_\kk(\kk-6)$ from $x$ to $z$. The law of $\eta$ does not depend on the choice of the points $(y_n)$.

\subsection{The Gaussian free field}\label{subsec:gff}
\subsubsection{The Gaussian free field on a simply connected domain}
Let $D \subsetneq \C$ be a simply connected domain and let $C_c^\infty(D)$ be the set of smooth, compactly supported functions on $D$. For $f,g \in \cci(D)$ we define their \textit{Dirichlet inner product} to be 
\begin{equation}\label{eq:dirichlet_norm}
    (f,g)_\nabla = \frac{1}{2\pi}\int_D \nabla f(z) \cdot \nabla g(z) \,dz.
\end{equation}
Let $H(D)$ be the closure of $\cci(D)$ with respect to \textit{Dirichlet norm} induced by this inner product. 
The Gaussian free field (GFF) on $D$ with zero boundary conditions is then defined as the (formal) series
\begin{equation}
    h = \sum_{n=1}^\infty \aal_n f_n
\end{equation}
where $(f_n)$ is an orthonormal basis of $H(D)$ with respect to the Dirichlet norm, and $(\aal_n)$ is a collection of i.i.d.\ standard normal random variables. A.s.\ this sum does not converge in $H(D)$, but does converge as a distribution. The law of $h$ does not depend on the choice of basis. 
If $\fh$ is a harmonic function on $D$ then the GFF with boundary conditions $\fh$ on $D$ is defined to be $h + \fh$, where $h$ is a zero-boundary GFF. 
If $f$ is a bounded measurable function on $\del D$, then it has a unique harmonic extension $\fh$ on $D$, so it suffices in this case to specify the values of $\fh$ on $\del D$.
We refer to \cite{bp_gff, she2007gff} for further details.

\subsubsection{The whole-plane GFF}
In the case that $D = \C$, the definition of the GFF is slightly more complex, and we are led to first introduce the \textit{whole-plane GFF modulo an additive constant}. Following \cite{ms2017ig4}, let $\cci(\C)$ be the set of smooth functions on $\C$ with compact support and mean $0$ and let $\cciz(\C)$ be the subset of $\cci(\C)$ for which $\int f dx = 0$. For $f,g \in \cci(\C)$ define their Dirichlet norm $(f,g)_\nabla$ by \eqref{eq:dirichlet_norm} with the integral over $\C$. We define $H(\C)$ to be the closure of $\cci(\C)$ with respect to the corresponding Dirichlet norm, modulo an additive constant. Let $(f_n)$ be an orthonormal basis for $H(\C)$, let $(\aal_n)$ be a collection of i.i.d.\ standard normal random variables and define
\begin{equation}
    h = \sum_{n=1}^\infty \aal_n f_n.
\end{equation}
As before, $h$ a.s.\ does not converge in $H(\C)$, but does converge as a distribution modulo an additive constant. That is, a.s.\ simultaneously for all $f \in \cciz(\C)$, the series
\begin{equation}
    \sum_{n=1}^\infty \aal_n (f_n, f)
\end{equation}
converges to a value which we call $h(f)$, and the map $h \colon f \mapsto h(f)$ is a continuous linear functional on $\cciz(\C)$. We refer to $h$ as the whole-plane GFF modulo an additive constant.

To define $h$ as a random distribution, as opposed to a distribution modulo an additive constant, we have to fix this constant in some way. This can be done by choosing a fixed element $g_0 \in H(\C)$ with $\int g_0 dz = 1$ and requiring that $h(g_0) = 0$. 
For any other $f \in \cci(\C)$ we then define $h(f) = h(f + ag_0)$ where $a \in \R$ is chosen such that $f + ag_0 \in \cciz(\C)$. This allows us to view $h$ as a distribution (instead of a distribution modulo additive constant) but the law of $h$ will then depend on how we choose to fix the additive constant.
We may also define $h$ as a \textit{whole-plane GFF modulo $r$}, where $r > 0$. This is done by letting $h$ be a GFF modulo an additive constant, sampling a uniform random variable $U$ on $[0, r)$, and requiring that $h(g_0) \equiv U \pmod r$. We can then define $h(f)$ as a value $\mod r$ using $g_0$. 
In this case $h$ is a distribution modulo $r$. We refer to \cite{ms2017ig4} for a more detailed explanation.

The GFF (in both the whole-plane and non--whole-plane case) is regular enough that its average on a given circle can be defined a.s. 
That is, $h(\rho_{z,\ee})$ exists a.s.\ where $\rho_{z,\ee}$ denotes the uniform measure on $\del B(z,\ee)$ with total measure $1$ for some $z \in \C$ and $\ee > 0$ (see e.g.\ \cite{bp_gff}).
When $h$ is a whole-plane GFF the additive constant must be fixed in some way for this to make sense, however.
It turns out that we are able to fix the additive constant of the whole-plane GFF in such a way that $h(\rho_{z,\ee}) = 0$. Then, $h$ is a random distribution, defined on all elements of $\cci(\C)$, where $h(\rho_{z,\ee})$ is defined and equal to $0$.
A similar procedure is explained in \cite[Section~6.3]{bp_gff}. 
Similarly, for $r > 0$ we can let $U$ be a uniform random variable on $[0,r)$ and require that $h(\rho_{z,\ee}) = U \pmod r$. Under this normalisation $h$ is a whole-plane GFF modulo $r$.
Finally we can also define the whole-plane GFF so that $h(\rho_{z,\ee})$ is uniform in $[0,r)$. This involves taking $U$ as before and then requiring that $h(\rho_{z,\ee}) = U$. 
Note that $h$ is now a distribution, not a distribution modulo $r$. In the setting of imaginary geometry (introduced below) it is natural to consider the whole-plane GFF as a distribution modulo $2\pi\chi$, or, alternatively, normalised so that $h(\rho_{0,1})$ is uniform in $[0, 2\pi\chi)$.

\subsubsection{Markov property and local sets}
The GFF has a natural Markov property \cite{bp_gff, ms2017ig4}. Let $A$ be a closed subset of the domain $D$ (so that $D\sm A$ has regular boundary --- see \cite{bp_gff}). 
The GFF can then be written as $h = \fh + h^0$, where $\fh$ is a random distribution in $D$ which is a.s.\ equal to a harmonic function on $D \sm A$, $h^0$ is a random distribution with the law of a zero-boundary GFF on $D \sm A$, and $\fh$ and $h^0$ are independent.
In the case that $D = \C$, and $h$ is a whole-plane GFF, then the same statement holds, but if $h$ is viewed as a distribution modulo an additive constant (or modulo $r$), then $\fh$ is also viewed as a distribution modulo an additive constant (or modulo $r$).

A local set $A$ is a random closed subset of $D$ which can be coupled with a GFF $h$ such that there exist two distributions $h^0$ and $\fh$ such that $\fh$ is a.s.\ harmonic on $D \sm A$ and conditional on $A$ and $\fh$, $h^0$ has the law of a zero-boundary GFF on $D \sm A$.
The same statement holds when $D = \C$, up to $\fh$ possibly being defined modulo constants. 
Often in the cases we will consider, the distribution $\fh$ is determined entirely by the harmonic function $\fh|_{D \sm A}$, and we can view $\fh$ simply as a harmonic function.
In this case, we may be able to characterise $\fh$ by its \textit{boundary conditions} on $\del(D\sm A)$. We also sometimes refer to these as the boundary conditions of $h$ given $A$.
Local sets were first introduced in \cite{schramm2013contour}. We refer to this paper and to \cite{ms2016imag1, ms2017ig4} for further details.

\subsection{Imaginary geometry}
\textit{Imaginary geometry} is the general term given to the coupling between the GFF and SLE described in \cite{ms2016imag1, ms2017ig4}. Heuristically, it involves viewing SLE curves as \textit{flow lines} of a vector field $e^{ih/\chi}$ where $h$ is a GFF.
However, since $h$ is a distribution and not a function, the vector field $e^{ih/\chi}$ does not exist so the coupling has to be defined in a different way. 
We will give a brief overview here of some of the results of \cite{ms2016imag1, ms2017ig4}. We refer to those two papers for further details.

\subsubsection{Flow lines starting from the boundary}
Fix $\kk \in (0,4)$ and define 
\begin{equation}\label{eq:ig_constants}
    \kk' = \frac{16}{\kk},\quad \la = \frac{\pi}{\sqrt\kk},\quad \chi = \frac{2}{\sqrt\kk} - \frac{\sqrt\kk}{2},\quad \la' = \frac{\pi}{\sqrt{\kk'}} = \la - \frac{\pi}{2}\chi. 
\end{equation}
Suppose $x_R \geq 0$ and let $\rho > -2$. We refer to $x_R$ as a force point.
Let $h$ be a GFF on $\h$ with boundary conditions given by $-\la$ on $\R_-$, $\la$ on $(0, x_R)$ and $\la(1 + \rho)$ on $(x_R, \infty)$. We will mainly be interested in the case that $x_R = 0$ (we view the force point as lying infinitesimally to the right of $0$, and sometimes write $x_R = 0^+$).
Let $\eta$ be an $\SLE_\kk(\rho_R)$ with force point at $x_R$ with associated processes $W$ and $O$ satisfying \eqref{eq:chordal_force_point}, and for each time $t \geq 0$ let the hull $K_t$ be the complement of the unbounded connected component of $\eta([0, t])$. Let $f_t \colon \h \sm K_t \to \h$ be the unique conformal map sending $\eta(t)$ to $0$, fixing $\infty$ and with $f_t(z)/z \to 1$ as $z \to \infty$.
Let $\fh_t^0$ be the harmonic function on $\h$ with boundary conditions equal to $-\la$ on $\R_-$, $\la$ on $(0, O_t - W_t)$ and $\la(1 + \rho)$ on $(O_t - W_t, \infty)$.
Then \cite[Theorem~1.1]{ms2016imag1} states that $\eta$ can be coupled with $h$ such that for any stopping time $\tau$ (for the natural filtration of $(W,O)$), the set $\eta([0, \tau])$ is a local set of $h$, and $h$ can be decomposed as $h = \fh_\tau + h^0$, where $\fh_\tau = \fh_\tau^0 \circ f_\tau - \chi \arg f_\tau'(z)$ and $h^0$ is a zero-boundary GFF on $\h \sm K_t$.
It is also shown in \cite{ms2016imag1} that under this coupling, $\eta$ is measurable with respect to $h$. We refer to $\eta$ as the \textit{flow line (of angle zero)} of $h$ (starting from $0$ and targeted at $\infty$).

For $\kk' > 4$, suppose $h$ has boundary conditions $\la'$ on $\R_-$, $-\la'$ on $(0, x_R)$ and $-\la'(1 + \rho)$ on $(x_R, \infty)$. Then an $\SLE_{\kk'}(\rho)$ curve $\eta'$ can be coupled with $h$ in the same way as above, except that we redefine $\fh_t^0$ to be the harmonic function on $\h$ with boundary conditions equal to $\la'$ on $\R_-$, $-\la'$ on $(0, O_t - W_t)$ and $-\la'(1 + \rho)$ on $(O_t - W_t, \infty)$.
This coupling is again unique, and $\eta'$ is measurable with respect to $h$. We say that $\eta'$ is the \textit{counterflow line} of $h$.

Let $\tht \in \R$. We say that $\eta$ is a flow line of angle $\tht$ of a GFF $h$ on $\h$ if $\eta$ is a flow line of $h + \chi\tht$ as described above. We sometimes refer to the flow line above as the flow line of angle zero. 
Note that we need to make sure the boundary conditions of $h$ are such that the boundary conditions of $h + \chi\tht$ are of the form described above.
Flow (and counterflow) lines of angle $\tht$ can be defined in other simply connected domains $D \subsetneq \C$ via conformal mapping. We refer to \cite{ms2016imag1} for further details, although it is worthwhile to mention the following fact explicitly here. 
Let $\psi \colon D \to \h$ be a conformal map. Then if $\eta_\h$ is a (counter)flow line of a GFF $h_\h$ on $\h$, then $\eta := \psi\nv(\eta_\h)$ will be a (counter)flow line of the field $h = h_\h \circ \psi - \chi\arg\psi'$.
We refer to the transformation $(\h, h_\h) \mapsto (D, h = h_\h \circ \psi - \chi\arg\psi')$ as a \textit{transformation of imaginary surfaces}.

\subsubsection{Interior flow lines}
In \cite{ms2017ig4} it was explained how to construct flow lines started from interior points of the domain $D$, which we call \textit{interior flow lines}. We sometimes refer to flow lines started from the boundary as \textit{chordal} flow lines to distinguish them from interior flow lines.
We describe the interior flow line coupling now. Let $D$ be a simply connected domain and let $h$ be a GFF on $D$.
We allow the case that $D = \C$ but in this case $h$ is a whole-plane GFF viewed as a distribution modulo $2\pi\chi$.
As explained in \cite{ms2017ig4}, for a fixed $\kk \in (0,4)$ and $z \in D$ there exists a coupling between a continuous curve $\eta$, starting from $z$, and $h$ such that for any $\eta$-stopping time, the set $\eta([0,\tau])$ is a local set of $h$, where the conditional boundary conditions for $h$ given\footnote{Unlike in the chordal case, to determine the boundary conditions, we need more information than just the curve $\eta([0,\tau])$. We also need to know the values of $h$ on an infinitesimal neighbourhood around the curve. We refer to \cite{ms2017ig4} for a precise statement.} $\eta([0,\tau])$ are the so-called \textit{flow line boundary conditions} described in \cite[Theorem~1.1]{ms2017ig4}. 
Under this coupling, $\eta$ is stopped when it first hits $\del D$, but if $D$ has boundary conditions of a certain form, then $\eta$ can be extended by sampling (repeatedly, if necessary) chordal flow lines of the field conditional on $\eta$ up until this time.
In this coupling, $\eta$ is measurable with respect to $h$. In fact, if $h$ changes by a constant multiple of $2\pi\chi$ (when $D \neq \C$) then $\eta$ is left unchanged, so $\eta$ is measurable with respect to $h$, viewed as a distribution modulo $2\pi\chi$.
This gives some explanation as to why we consider $h$ as a distribution modulo $2\pi\chi$ in the whole-plane case (see \cite{ms2017ig4} for further details).
If $\tht \in \R$ we define $\eta$ to be a flow line of $h$ of angle $\tht$ if it is a flow line of $h + \chi\tht$.

As explained in \cite{ms2017ig4}, we can also define counterflow lines of a GFF $h$ targeted at interior points $z \in D$. Then, there is a \textit{duality} between flow and counterflow lines of $h$. 
It is proved in \cite{ms2017ig4} that if $\eta'$ is the counterflow line from $x \in \del D$ to $z \in D$, then a.s.\ the left and right boundaries of $\eta'$ are equal to the (interior) flow lines of $h$ started from $z$ of angle $\pi/2$ and $-\pi/2$ respectively.

To avoid repeating ourselves many times in the paper, we will use the following shorthand. If $h$ is a GFF on $\D$, then we will say that $h$ has boundary conditions given by $a$ \textit{to the left of $-i$} and by $a - 2\chi\pi$ \textit{to the right of $-i$} to mean that $h = h_\h \circ \psi - \chi\arg(\psi')$ where $h_\h$ is a GFF on $\h$ with boundary conditions given by $a$ on $\R_-$ and by $a - 2\chi\pi$ on $\R_+$, and where $\psi$ is a conformal map from $\D$ to $\h$ sending $-i$ to $0$ (and say sending $i$ to $\infty$, although due to the specific choice of boundary conditions this is not actually necessary). 
Notice that the boundary conditions for $h$ on $\del \D$ are then continuous away from $-i$, and are indeed equal to $a$ and $a - 2\chi\pi$ at points infinitesimally close to $-i$ on its left and right sides respectively. When we draw figures, we will often replace $\D$ by a square to make it easier to denote boundary conditions. This can be safely ignored, and imagined to simply be a disk.

\section{Conformal loop ensembles and space-filling SLE}\label{sec:cle_sfsle}
In this section, we present constructions of $\CLE_\kk$ for $\kk \in (8/3, 4)$ and $\CLE_{\kk'}$ for $\kk' \in (4,8)$. 
We also define space-filling $\SLE_{\kk'}$ for $\kk' \in (4,8)$. 
We then prove two lemmas on the interaction between space-filling $\SLE_{\kk'}$ and $\CLE_{\kk'}$ (Lemma~\ref{lem:sfsle_pocket_interaction_48}), and between space-filling $\SLE_{\kk'}$ and $\CLE_{\kk}$ where $\kk' = 16/\kk$ (Lemma~\ref{lem:sfsle_pocket_interaction_834}).
This relationship will be the key to proving Theorem~\ref{thm:diam}.

\subsection{Conformal loop ensembles}\label{sec:cle}
In \cite{sw2012markovian}, a simple CLE in a simply connected domain $D \subsetneq \C$ is defined to be a locally finite, countable collection of simple, non-intersecting loops $\Gamma$ whose law, $P_D$, is conformally invariant and which satisfies a restriction property which we will now describe. By conformal invariance, we mean that if $D'$ is another simply connected domain, then $P_{D'}$ is equal to the image of $P_D$ under any conformal map from $D$ to $D'$.
This also implies that $P_D$ is invariant under conformal automorphisms.
The restriction property we require is as follows; if $D_1 \subset D$ and we define $D_1^*$ to be the subset of $D$ obtained after removing all loops which leave $D_1$ and their interiors, then the conditional law of the remaining loops, given the set of loops which leave $D_1$, is given by $P_{D_1^*}$. 
It is shown in \cite{sw2012markovian} that any simple CLE is given by a $\CLE_\kk$ for $\kk \in (8/3,4]$, which we construct below (for $\kk \neq 4)$.
A non-simple CLE in $D$ is similarly defined as a locally finite countable collection of (self-intersecting) loops in a domain where a.s.\ some loops do intersect the boundary. We refer to the conditions in \cite[Theorem~5.4]{s2009cle} for the full list of requirements. 
It was shown in \cite{s2009cle}, conditional on a conjecture on properties of $\SLE_{\kk'}(\kk'-6)$ curves which was later proved in \cite{ms2016imag1, ms2016imag3}, that any non-simple CLE is a $\CLE_{\kk'}$ for $\kk' \in (4,8)$. The loops in a non-simple CLE may intersect themselves and each other, but they do not cross in either case.

We will now describe the construction of conformal loop ensembles for $\kk \in (8/3, 8) \sm \{4\}$ using the \textit{continuum exploration tree} approach introduced in \cite{s2009cle}. The cases $\kk \in (8/3, 4)$ and $\kk' \in (4,8)$ are different, and will be explained independently of each other. We do not consider the $\kk = 4$ case in this paper so omit its construction.

\subsubsection{Construction of $\CLE_\kk$ for $\kk \in (8/3, 4)$}\label{subsec:cle834}
For $\kk \in (8/3, 4)$, the continuum exploration tree approach was first used in \cite{s2009cle} and revisited in \cite{werner_wu} and \cite{msw2017cleperc}. By conformal invariance, it will suffice to construct $\CLE_\kk$ on $\D$.
Let $(z_n)$ be the set of rational points (by which we mean a point with rational coordinates) in $\D$ and for each $n$ let $\eta_{z_n}$ be a radial $\SLE_\kk(\kk-6)$ process from $-i$ to $z_n$.
By the target invariance property of $\SLE_\kk(\kk-6)$ described in Section~\ref{subsubsec:gen_sle}, it is possible to couple the family $(\eta_{z_n})$ in such a way that for any two points $z_m, z_n$, the curves $\eta_{z_m}$ and $\eta_{z_n}$ are equal up until the time that $z_n$ and $z_m$ are separated by the curve, and such that the curves evolve independently after these times.
This construction provides us with a continuum branching tree structure which is called the continuum exploration tree.

Recall from Section~\ref{subsubsec:gen_sle} that by \cite{msw2017cleperc}, for a fixed point $z \in \D$, the process $\eta_{z}$ is a (self-intersecting but non-self-crossing) curve, and this curve consists of a ``trunk" which has the law of an $\SLE_{\kk'}(\kk'-6)$ (with force point at $0^+$), and a collection of excursion loops hanging off this trunk on its right side. 
It is explained in \cite{werner_wu} that, conditional on an initial segment of such an excursion loop, the remainder of the loop has the law of a standard $\SLE_\kk$ curve.
We refer to \cite{werner_wu} and \cite{msw2017cleperc} for a more detailed explanation.

We define the CLE loop $\ell_z$ to be the first excursion loop of $\eta_z$ which surrounds $z$. Note that this is not simply the first time that $\eta_z$ traces a path around $z$ (which could be caused by the trunk surrounding $z$); the loop in question must be one of the excursion loops of $\eta_z$ corresponding to an excursion of $W_t - O_t$ away from $0$. 
The loop $\ell_z$ exists almost surely, and we define our $\CLE_\kk$, which we refer to as $\Gamma$, to be the collection of loops $(\ell_{z_n})$. A.s.\ $\Gamma$ is a collection of simple, non-intersecting loops which do not intersect $\del \D$.

Next, we explain based on the work in \cite{msw2017cleperc} how this fits naturally into the imaginary geometry framework of \cite{ms2016imag1, ms2017ig4}. 
Let $h$ be a GFF on $\D$ with boundary conditions given by $\la'$ to the left of $-i$ and $\la'-2\chi\pi$ to the right of $-i$ (recall that this is explained in Section~\ref{subsec:gff}). 
Then (using \textit{boundary conformal loop ensembles}, which we will not describe here) it is shown in \cite{msw2017cleperc} that for $y \in \del \D$, the chordal $\SLE_\kk(\kk-6)$ from $-i$ to $y$, $\eta_y$, can be coupled with the GFF $h$ so that for any $\eta_y$-stopping time $\tau$, the set $\eta_y([0, \tau])$ is a local set for $h$ such that the boundary conditions for $h$ given $\eta_y([0, \tau])$ are as in Figure~\ref{fig:eta_y_bcs}. 
We say that $\eta_y$ is coupled with $h$ as a \textit{generalised flow line} ($\eta_y$ is coupled to $h$ in a very similar way to the flow and counterflow lines of \cite{ms2016imag1}, but $\eta_y$ is neither a flow nor counterflow line, but a different type of curve coupled with the GFF).
Furthermore, under this coupling, $\eta_y$ is a function of $h$.

\begin{figure}[t]
    \centering
    \includegraphics[scale=1]{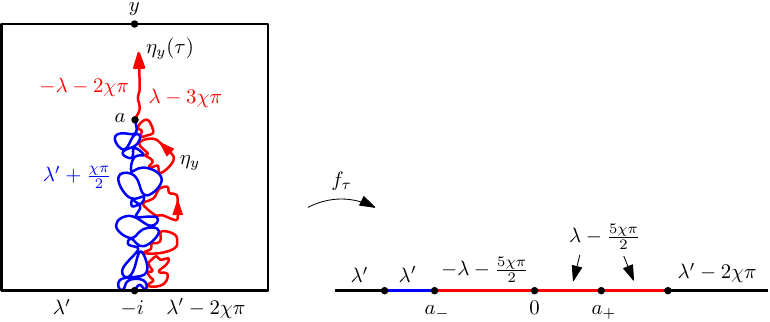}
    \caption{
    Let $h$ be a GFF on $\D$ (which we depict as a square to make it easier to denote boundary conditions) with boundary conditions given by $\la'$ to the left of $-i$ and $\la' - 2\chi\pi$ to the right. 
    Let $\eta_y$ be an $\SLE_\kk(\kk-6)$ from $-i$ to $y$ coupled with $h$ as described in Section~\ref{subsec:cle834}, and let $\tau$ be a stopping time for $\eta_y$ such that $W_\tau \neq O_\tau$ a.s. 
    The trunk (blue) and excursion loops (red) of $\eta_y$ are drawn in different colours to distinguish them.
    Then, by the arguments of \cite[Section~9]{msw2017cleperc}, $\eta_y([0,\tau])$ is a local set for $h$ and its boundary conditions can be described as follows.    
    Let $f_\tau$ be a conformal map from the component of $\D$ with $y$ on its boundary to $\h$ which sends $\eta_y(\tau)$ to $0$.
    Let $\aal$ be the most recent time before $\tau$ where $O_t = W_t$ and let $a = \eta_y(\aal)$. Then on the time interval $[\aal, \tau]$, $\eta_y$ traces part of an excursion loop, which is simple and intersects $\eta_y([0,\aal])$ only at $a$. 
    It follows that there are two prime ends corresponding to the point $a$, which we say are sent to $a_-, a_+$ on $\R$ so that $a_- < 0 < a_+$.
    Let $\fh_1$ be the harmonic function on $\h$ with boundary conditions given by $\la'$ on $(-\infty, a_-)$, $-\la - 5\chi\pi/2$ on $(a_-, 0)$, and $\la - 5\chi\pi/2 = \la' - 2\chi\pi$ on $(0, \infty)$. Then the boundary conditions for $h$ given $\eta_y([0, \tau_y])$ are given by $\fh = \fh_1 \circ f_\tau - \chi\arg(f_\tau')$.
    The boundary conditions in colour (red and blue) on the left hand figure correspond to what the boundary conditions of $\fh$ would be on a (hypothetical) vertical segment of the corresponding curve. This convention is used throughout the paper.}
    \label{fig:eta_y_bcs}
\end{figure}

If instead we have $z$ in the interior of $\D$, we can couple $\eta_z$ with $h$ using the following procedure, depicted in Figure~\ref{fig:eta_z_construction}.
Set $y_1 = i \in \del \D$ and let $\tau_1$ be the first time that $\eta_1$, the chordal $\SLE_\kk(\kk-6)$ curve from $-i$ to $y_1$ which is coupled to $h$ as a generalised flow line as described above, disconnects $0$ from $i$. 
Let $D_2$ be the connected component of $\D \sm \eta_1([0, \tau_1])$ containing $0$ and let $h_2$ be the restriction of $h$ to $D_2$. We can determine the boundary conditions on $\del D_2$ of $h_2$ given $\eta([0, \tau_1])$ explicitly as in Figure~\ref{fig:eta_y_bcs}. Indeed, if $\psi_2 \colon D_2 \to \h$ is a conformal map sending $\eta_y(\tau_1)$ to $0$ and $z$ to $i$, say, then $h_2$ has boundary conditions given by $\fh_2 = \wt\fh_2 \circ \psi_2 - \chi\arg(\psi_2')$ where $\wt\fh_2$ is equal to $\la'$ on $\R_-$ and $\la' - 2\chi\pi$ on $\R_+$. That is, the conditional boundary conditions of $h_2$ are exactly the same as those of $h$ (up to a transformation of imaginary surfaces).
This means that we can define an $\SLE_\kk(\kk-6)$ curve $\eta_2$, from $\eta_1(\tau_1)$ targeted at a point $y_2$ (chosen in an arbitrary measurable way) on $\del D_2$, which is coupled with $h_2$ as a generalised flow line.
Let $\tau_2$ be the first time $\eta_2$ disconnects $0$ from $y_2$.
We continue in this manner and construct a sequence of curves $(\eta_k([0, \tau_k]))_{k \geq 1}$. Appending this countable collection of curves together gives a radial $\SLE_\kk(\kk-6)$ process $\eta_z$ from $-i$ to $z$ which is a measurable function of $h$. 
For any stopping time $\tau$ for $\eta_z$ (where we have chosen some parameterisation for $\eta_z$), we have that $\eta_z([0, \tau])$ is a local set for $h$, and we can determine the corresponding boundary conditions using the known boundary conditions of the curves $(\eta_k)$. 
We can perform this process for all rational points $(z_n)$, and thus determine a collection of curves $(\eta_{z_n})$, which are coupled in such a way that $\eta_{z_m} = \eta_{z_n}$ (up to parametrisation) up until the point that $\eta_{z_m}$ separates $z_m$ from $z_n$, and are independent thereafter.
That is, this coupling of the curves with $h$ leads exactly to the situation of the continuum exploration tree described above. 
Furthermore, these curves are all measurable with respect to $h$, meaning that we can view our $\CLE_\kk$, $\Gamma$, as being coupled with the GFF $h$ such that $\Gamma$ is a measurable function of $h$.

\begin{figure}[t]
    \centering
    \includegraphics[scale=1]{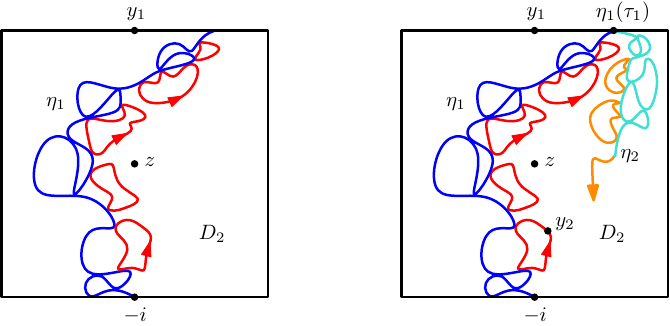}
    \caption{
    Here we show the first two steps in the construction of $\eta_z$, a radial $\slekk$ in $\D$ from $-i$ to $z$.
    $\eta_1$ is a chordal $\slekk$ from $-i$ to $y_1$ stopped when it first disconnects $z$ from $y_1$, at time $\tau_1$. 
    $\eta_2$ is a chordal $\slekk$ from $\eta_1(\tau_1)$ to $y_2$ in $D_2$, stopped when it first disconnects $z$ from $y_2$.
    We define $\eta_k$ for $k \in \N$ inductively in this fashion.
    As explained in Section~\ref{subsec:cle834}, the curves $\eta_k$ can be coupled with the GFF $h$.
    In particular, $\eta_k$ is coupled to the restriction of $h$ to $D_k$ as a generalised flow line.}
    \label{fig:eta_z_construction}
\end{figure}

\subsubsection{Construction of $\CLE_{\kk'}$ for $\kk' \in (4,8)$}\label{subsec:cle48}
Fix $\kk' \in (4,8)$. We will describe the construction of $\CLE_{\kk'}$ again using the continuum exploration tree, as in \cite{s2009cle}. As before, we may assume the domain we are working on is $\D$.
For each $z \in \D$, let $\eta'_y$ be a radial $\sleu$ from $-i$ to $z$ with force point at $-ie^{i0^+}$ (as explained in Section~\ref{subsec:cle48}). For $y, z \in \D$ distinct, the curves $\eta'_y$ and $\eta'_z$ have the same law up until the time that $y$ is disconnected from $z$ \cite[Proposition~3.14]{s2009cle}.
Let $(z_n)$ be a countable dense set of points in $\D$, and consider the curves $(\eta'_n)$ coupled together such that $\eta'_n$ and $\eta'_m$ agree up until the time that $z_n$ and $z_m$ are disconnected. 
One way to create this coupling is to let $h$ be a GFF on $\D$ with boundary conditions equal to $\la'$ to the left of $-i$ and $\la' - 2\chi\pi$ to the right of $-i$. Then if we let $\eta'_n$ be the radial counterflow line of $h$ from $-i$ to $z_n$ (as defined in \cite{ms2017ig4}), the curves $\eta'_n$ will naturally be coupled as above.

Next, we explain how to construct (non-nested) $\CLE_{\kk'}$ from the continuum exploration tree, which we continue to view as a countable collection of flow lines of $h$. $\clekp$ a.s.\ consists of a countable number of loops, each of which contains some point $z_n$, so it suffices to construct the loop $\CL_z$ surrounding a fixed point $z \in \D$.
Our description is based on those in \cite[Section~4]{s2009cle} and \cite[Section~2.2]{msw2014gasket}.
Each curve $\eta_z'$ is a radial $\SLE_{\kk'}(\kk'-6)$ curve, and has associated to it processes $W^z, O^z$ and $\tht^z$ as introduced in Section~\ref{subsubsec:radial_sle} (these processes are obtained by conformally mapping $z$ to $0$ using a map $\ff \colon \D \to \D$ and then considering the curve $\ff(\eta_z)$).
Note that $\tht^z_0 = 2\pi$.
Let $\tau_z = \inf\{t \geq 0 \colon \tht_z(t) = 0\}$; this is the first time that $\eta'_z$ makes a clockwise loop around $z$, with $\tau_z = \infty$ if this does not occur. If $\tau_z = \infty$, then $z$ is not surrounded by any such loop, or equivalently $z \in \Uu$, so $\CL_z$ is not defined. This occurs with probability $0$ for a fixed $z$, and we will now work on the event that $\tau_z < \infty$.
Let $\tau'_z = \sup\{t < \tau_z \colon \tht_z(t) = 2\pi\}$. Let $\gg_z$ be the curve $\eta'_z([\tau'_z, \tau_z])$. For any other point $w \in \D$, we can similarly define $\gg_w$, and it is then the case that either $\gg_w$ contains $\gg_z$ as a sub-arc, $\gg_z$ contains $\gg_w$ as a sub-arc, or neither of these properties hold. We define $\CL_z$ to be the union of $\gg_{z_n}$ for those $n$ where $\gg_z$ is a sub-arc of $\gg_{z_n}$.
In this paper, we will not be interested precisely in the loop $\ell_z$ surrounding a point $z$, but instead in the connected component, or \emph{pocket}, of $\D \sm \Uu$ containing $z$.
Note that to merely construct the pocket containing $z$ using the above process is simpler, since it is determined only by $\gg_z$ itself and we do not have to worry about how to complete the loop.

We conclude this section by proving that the $\CLE_{\kk'}$ gasket $\Uu$ is a.s.\ a Sierpi\'nski carpet for any $\kk' \in (4,8)$.
Note that the terminology $\CLE_{\kk'}$ gasket is used in reference to the fact that the $\CLE_{\kk'}$ \textit{loops} a.s.\ intersect themselves and each other.
If we instead consider the connected components, or pockets, of $\D\sm\Uu$ however, we will see that these are a.s.\ Jordan domains whose boundaries do not intersect each other.
We remark that this proposition is not needed for the proof of Theorem~\ref{thm:diam}, but is necessary for the application of Ntalampekos's uniformization result in \cite{ntalampekos_conformal_uniformization}.

\begin{proposition}\label{prop:gasket_is_carpet}
For $\kk' \in (4,8)$, the $\CLE_{\kk'}$ gasket $\Uu$ is a.s.\ a Sierpi\'nski carpet.
\end{proposition}

\begin{proof}
\textit{Setup.}
Recall that a Sierpi\'nski carpet is a closed subset of $\wh\C$ with empty interior and whose complement consists of countably many Jordan domains whose boundaries do not intersect and whose diameters shrink to $0$ in the spherical metric.
That $\Uu$ has empty interior follows from the fact that fixed points are a.s.\ not in the set and the fact that the diameters shrink to $0$ follows from our main result, Theorem~\ref{thm:diam} (which does not rely on this lemma).
It remains to prove that a.s.\ the boundaries of two distinct complementary components do not intersect, which we will do in Steps 1 to 4, and that the complementary components are Jordan domains, which we will do in Step 5.
Our main input in showing that the boundaries of distinct components do not intersect is Lemma~2.5 of \cite{gp2020adj}, which states that if $\eta'$ is a chordal $\SLE_{\kk'}$ (with no force points) no two connected components of $\D \sm \eta'$ which lie on the same side of $\eta'$ intersect each other a.s., although we will have to work through some other cases before applying it to complete the proof.
First we convert the above result to $\SLE_{\kk'}(\kk' - 6)$ and then we show how this implies our result for the $\CLE_{\kk'}$ gasket.

\textit{Step 1. Translating a result for $\SLE_{\kk'}$ to $\sleu$.} 
Let $\eta'_1$ be a $\SLE_{\kk'}(\rho_L;\rho_R)$ curve on $\D$ from $-i$ to $i$, with $\rho_L, \rho_R \geq \kk'/2 - 4$ (we refer to \cite{ms2016imag1} for the definition of SLE with multiple force points).
Let $\eta_L$ and $\eta_R$ be the left and right boundaries of $\eta'_1$. By \cite{ms2016imag1}, one can show that in each component lying between $\eta_L$ and $\eta_R$, $\eta'_1$ has the law of an $\SLE_{\kk'}(\kk'/2 - 4, \kk'/2 - 4)$ from the opening point of the component to the closing point. 
Here, we view $\eta_L, \eta_R$ as running from $-i$ to $i$ and define the opening (resp.\ closing) point of a component lying between them to be the first time (resp.\ second, or equivalently last) time they intersect on the boundary of this component---note that $\eta_L$ and $\eta_R$ are a.s.\ simple.
It follows that if there is a positive probability that two components on the same side of an $\SLE_{\kk'}(\kk'/2 - 4, \kk'/2 - 4)$ intersect each other, then there is a positive probability that two components on the same side of $\eta'_1$ intersect each other.
Furthermore, no complementary components of $\eta'_1$ which lie in different components $U,V$ lying between $\eta_L$ and $\eta_R$ can intersect each other. Indeed this follows since any two such components $U$ and $V$ themselves do not intersect each other. 
If this were the case, then (assuming without loss of generality that $U$ is hit by $\eta_L$ and $\eta_R$ first) the closing point of $U$ would have to be the opening point $V$ (since $\eta_L$ and $\eta_R$ are simple curves).
But after $\eta_L$ and $\eta_R$ intersect at $U$, they will a.s.\ intersect again in any arbitrarily small neighbourhood of this closing point, meaning that this cannot be the opening point of a different pocket.
Therefore we can conclude that there is a positive probability that two components on the same side of $\eta'_1$ intersect if and only if there is a positive probability that two components of the same side of an $\SLE_{\kk'}(\kk'/2 - 4, \kk'/2 - 4)$ intersect each other.
Since we know the former event has probability $0$ for $\rho_L = \rho_R = 0$, it follows that it has probability $0$ for $\rho_L = \rho_R = \kk'/2 - 4$ and hence for all $\rho_L, \rho_R \geq \kk'/2 - 4$. In particular it has probability $0$ for an $\SLE_{\kk'}(\kk'-6)$.

\textit{Step 2. Components to the right of a chordal $\sleu$ do not intersect the left boundary of a domain.}
Before moving on to $\CLE_{\kk'}$, we will first prove the following. Let $\eta'$ be a chordal $\sleu$ on $\D$ from $-i$ to $i$. Then any the boundary of any component of $\D \sm \eta'$ which lies to the right of $\eta'$ a.s.\ does not intersect $\del^L \D$, the clockwise arc of $\del \D$ from $-i$ to $i$. Define $\del^R \D$ analogously.
To prove this result, let $h$ be a GFF on $\D$ with boundary conditions $\la'$ to the left of $-i$ and $\la' - 2\pi\chi$ to its right.
Then by the flow line duality  results of \cite[Theorem~1.4]{ms2016imag1} the left and right boundaries of $\eta'$, denoted by $\eta_L$ and $\eta_R$, are flow lines of angle $-3\pi/2$ and $-\pi/2$ respectively. 
One can then compute that $\eta_R$ has the law of an $\SLE_\kk(-\kk/2; \kk/2 - 2)$ from $i$ to $-i$ and hence will a.s.\ intersect $\del^R \D$ but not $\del^L \D$ \cite[Lemma~2.1]{mw2017intersections}.

Suppose that $P$ is any component of $\D \sm \eta'$ which is on the right side of $\eta'$. This includes components to the right of $\eta_R$ and also components which are disconnected from $i$ when $\eta'$ makes a clockwise loop.
The case that $P$ is to the right of $\eta_R$ is not relevant to us, but the result follows from the fact that $\eta_R$ a.s.\ intersects $\del^R \D$ arbitrarily close to both $-i$ and $i$.
Suppose then that $P$ is a complementary component of $\eta'$ which is separated from $i$ by $\eta'$ making a clockwise loop. Then $P$ will lie between $\eta_L$ and $\eta_R$.
By duality for interior flow lines (\cite[Theorem~1.13]{ms2017ig4}) and by taking a sequence of rational points approaching the point where this clockwise loop is completed, one can show that $\del P$ is contained in the union of flow lines of angle $-\pi/2$ of $h$ started from rational points in $\D$ which lie between $\eta_L$ and $\eta_R$.

Let $z \in \D$ be a fixed point and suppose we are on the event that $z$ lies between $\eta_L$ and $\eta_R$. Let $\eta_z$ be the flow line of $h$ started from $z$ (and targeted at $-i$) of angle $-\pi/2$.
We will show that $\eta_z$ cannot intersect $\del^L \D$.
Since $\eta_z$ cannot cross $\eta_L$, in order for this to occur, $\eta_z$ must intersect $\del^L \D$ at a point where $\eta_L$ intersects $\del^L \D$. 
This cannot occur since in order for $\eta_z$ to intersect $\eta_L$ it must do so with a height gap of at least $\pi$, but one can compute that at this height $\eta_z$ cannot intersect $\del^L \D$ (we refer to \cite{ms2017ig4} for an explanation of the height gap, and in particular to Theorem~1.7 of that paper for an explanation of the flow line interaction rules).
It follows that $\eta_z$ a.s.\ does not intersect $\del^L \D$, and hence that $\del P$ does not intersect $\del^L \D$ a.s.

\textit{Step 3. Applying these results to $\CLE_{\kk'}$ to prove the lemma: the case of counterclockwise loops.}
Suppose $w, z \in \D$ are fixed. We will show that a.s.\ the boundaries of $P_w$ and $P_z$ do not intersect. We break the proof into multiple cases depending on the behaviour of the curves $\eta'_w$ and $\eta'_z$. For legibility, we will abbreviate clockwise by CW and counterclockwise by ACW (for anticlockwise, since CW and CCW are too similar).

First, suppose that $w$ and $z$ are first separated when $\eta'_w$ makes an ACW loop $\CL$ (which may include part of $\del \D$) around $w$. Let $Q_w$ be the complementary component of $\D \sm \CL$ containing $w$.
Note that $P_z$ will necessarily lie outside $\CL$. 
We will now show that $\del P_w$ does not intersect $\del Q_w$, and hence does not intersect the loop $\CL$.
The remainder of $\eta'_w$ has the law of an $\sleu$ inside $Q_w$ from $x$ to $w$, where $x$ is the closing point of loop $\CL$ (this is the point at which the loop is completed; note that this point must be on $\del Q_w$ by definition).
Suppose first that $\eta'_w$ makes a CW loop around $w$ before making a further ACW loop around it.
Then, it follows that there is some point $y \in \del Q_w$ (chosen from some fixed countable dense subset of $\del Q_w$) such $\eta'_w$ agrees with the chordal $\sleu$ in $Q_w$ from $x$ to $y$, which we denote by $\eta'_y$, up until the point $\eta_y'$ draws this CW loop. 
Since $P_w$, the component containing $w$, is on the right side of $\eta_y$, its boundary does not intersect the CW arc of of $\del Q_w$ from its opening point to $y$ by Step 2 (this arc corresponds to $\del^L \D$ above).
Also, $\del P_w$ could intersect the ACW arc from $x$ to $y$ only if $\eta_y$ intersect this arc with $w$ on its right side, thus separating $w$ and $y$ at this point. 
By our choice of $y$, this could only occur if the closing point of $P_w$ itself was part of $\del Q_w$, 
but this a.s.\ does not occur. 
One way to see this is to show that whenever $\eta'_y$ travels some distance $\dd$ from this ACW arc of $\del Q_w$, it a.s.\ does not first hit this arc again at the most recent point at which $\eta'_y$ has intersected this arc. This follows from the fact that $\eta'_y$ fills neither $\del Q_w$ nor its own right boundary.
In conclusion, $\del P_w$ does not intersect $\del Q_w$ in this case.

Suppose now instead that $\eta'_w$ does make a further ACW loop around $z$ before making a CW one.
A.s. $\eta'_w$ makes only a finite number of ACW loops around $w$ before making a CW one, so if we let $Q'_w$ be the complementary component of this final ACW loop which contains $w$, then we can repeat the above argument to show that $\del P_w$ does not hit $\del Q_w'$ and hence does not intersect $\del Q_w$ either.
This ensures that $\del P_w$ does not intersect $\del Q_w$ and hence will not intersect $\del P_z$.

Note also that the case that $w$ and $z$ are first separated by $\eta'_w$ making a CW loop around $w$ which includes part of the boundary can equivalently be viewed as $\eta'_z$ making an \textit{ACW} loop around $z$ which includes a different part of the boundary, so this case is also included here. We may also reverse the roles of $w$ and $z$ in each case.

\textit{Step 4. The case where $w$ and $z$ are not separated by a counterclockwise loop.} 
The only remaining case is that $w$ and $z$ are separated when $\eta'_w$ makes a CW loop around $z$ which does not include part of the boundary (we choose this point to be $z$ instead of $w$ without loss of generality).
It is possible that $\eta'_w$ has made multiple ACW loops around both $w$ and $z$ before this time, but it can only make a finite number of such loops a.s. 
In that case, each time an ACW loop is made, the law of the $\CLE_{\kk'}$ on $\D$ inside the component of the loop containing $w$ and $z$ is exactly that of a $\CLE_{\kk'}$ on this component, so by considering the component $Q_w'$ of the last such ACW loop which contains $z$, as above, we can always reduce to the case that no such ACW loops are made.
That is, we have reduced to the case that $z$ and $w$ are first separated when a CW loop is first drawn around $z$, and that before this time $z$ and $w$ have not been disconnected from $\del \D$ by $\eta'_w$ making an ACW loop around them.
If $\eta'_w$ now goes on to draw an ACW loop $\CL$ around $w$, we can argue as above that $\del P_w$ does not intersect the boundary of the complementary component of $\CL$ containing $w$ (which is a different component to that containing $z$) and hence does not intersect $\del P_z$.
Finally, we can now assume that $\eta'_w$ first traces a CW loop around $z$, and then a CW loop around $w$ without any ACW loops being drawn around either point beforehand. 
In this case, there exists a point $q \in \del \D$ in some deterministic fixed countable dense subset of $\del \D$ such that $\eta'_w$ agrees with $\eta'_q$ up until the CW loop around $w$ is drawn. 
But $\eta'_q$ is a chordal $\sleu$, and the pockets $P_z$ and $P_w$ are both on the right side of this curve (by virtue of them being cut out by CW loops) so by Lemma~2.5 of \cite{gp2020adj}, their boundaries do not intersect.

\textit{Step 5. The complementary components are Jordan domains.}
To complete the proof it remains to show that each complentary component of $\Uu$ is a.s.\ a Jordan domain. For fixed $z \in \D$, we will show that this is the case for $P_z$ a.s.
In Step 3, we showed that if $z$ is first surrounded by an ACW loop, then $\del P_z$ a.s.\ does not intersect this loop, so by considering the complementary component $Q_z'$ of the final ACW loop of $\eta'_z$ surrounding $z$, we may reduce to the case that $\eta'_z$ makes no ACW loops around $z$.
Recall that $\tau_z$ is defined as the first time at which $\eta'_z$ makes a CW loop around $z$ and let $\ss_z = \inf\{t \geq 0 \colon \eta'_z(t) = \eta'_z(\tau_z)\}$.
Let $u \in \D$ be any point with rational coordinates. Let $\bb_u$ be the first time $u$ is separated from $z$ by $\eta'_z$.
A.s.\ there exists such a $u$ for which $\bb_u \in (\ss_z, \tau_z)$ and for which $u$ is disconnected from $z$ by $\eta'_z$ making a CW loop around $u$. In this case $\tau_u = \bb_u$ and we can define $\ss_u$ analogously to $\ss_z$. 
There further exists such a $u$ for which $\ss_u \in (\ss_z, \tau_z$); this is simply because of the fact that $\eta'_z$ will always draw arbitrarily small CW and ACW loops as it travels through $\D$. 
For the remainder of the proof we work on the event that the point $u$ satisfies these requirements.

See Figure~\ref{fig:pocket_is_jordan} for a depiction of the following. Let $\eta_u^R$ be the flow line of $h$ of angle $-\pi/2$ started from $u$. By interior duality \cite[Theorem~1.13]{ms2017ig4}, $\eta_u^R$ is a.s.\ the right boundary of $\eta'_u$, the counterflow line of $h$ targeted at $u$. 
The curve $\eta_u^R$ is a.s.\ not simple, but it can only intersect itself after it has wrapped around $z$ \cite{ms2017ig4}. Since $\eta_u^R$ is the right boundary of $\eta'_u$, at any self-intersection point of $\eta_u^R$ we must have that $\eta'_u$ hits this point (at least) twice, and that $\eta'_u$ must make a loop around $u$ between these two times. 
It follows that after $\eta'_u$ has made such a loop, $\eta_u^R$ (and hence also $\eta'_z$) has separated $u$ from $z$ (since we have assumed that the first loop made around $u$ does not also contain $z$).
The upshot of this is that the segment of $\eta^R_u$ from the last time it hits $\eta'_z(\tau_u)$, call this time $\ss$, until it terminates at $0$ is a simple curve.

\begin{figure}[t]
    \centering
    \includegraphics[scale=1]{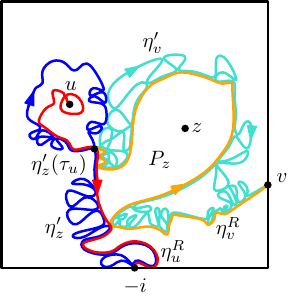}
    \caption{We depict here the setup in Step 5 of the proof of Proposition~\ref{prop:gasket_is_carpet}. $P_z$ is the pocket of $\Uu$ containing $z$ and is seen to lie between the two simple curves $\eta_u^R([\ss, \infty))$ and $\eta_v^R$, ensuring that $P_z$ is a.s.\ a Jordan domain.}
    \label{fig:pocket_is_jordan}
\end{figure}

Note that $\tau_u$ is a stopping time for $\eta'_z$. By assumption, $\eta'_z$ makes a CW loop around $z$ before an ACW one, so it follows that there a.s.\ exists a point $v \in \del \D$ such that $\eta'_z$ and $\eta'_v$ agree until this CW loop is made. 
Consider now the remainder of the curve $\eta'_v$ after the time $\tau_z$ (where the parameterisations of $\eta_u'$ and $\eta_v'$ agree). 
It is the counterflow line of the conditional field $h$ given $\eta'_v([0, \tau_u])$ from $\eta_v'(\tau_u)$ to $v$. As such, using the chordal flow line duality of \cite{ms2016imag1}, the right boundary $\eta_v^R$ of the remainder of $\eta'_v$ is the flow line of this conditional field of angle $-\pi/2$ from $v$ to $\eta'_v(\tau_u)$ and is therefore a simple curve. 
Furthermore, by considering the law of this curve (which we can identify using the boundary conditions of the conditional field), we see that it a.s.\ does not intersect the left boundary of $\eta'_z([0, \tau_u])$. 
Since we have assumed that $\eta'_v$ makes a CW loop around $z$, it must be the case that $\del P_z$ is contained in the union of $\eta_u([\ss, \infty))$ and $\eta_v^R$, which are two simple curves which intersect only by $\eta_v^R$ hitting the left side of $\eta_u([\ss, \infty))$ (where we view $\eta_u$ as travelling in the direction from $u$ to $-i$). 
The pocket $P_z$ will then be a region between these two curves, as shown in Figure~\ref{fig:pocket_is_jordan}, ensuring that $\del P_z$ is a simple curve and that $P_z$ is a Jordan domain a.s.
\end{proof}

\subsection{Space-filling SLE}\label{sec:sfsle}
Space-filling variants of SLE were introduced and their relationship to CLE was explained in \cite{ms2017ig4}. Here we will describe the space-filling $\slekp$ curves which will be used in this paper. Although this is not the formal definition, it is helpful to think of a space-filling $\slekp(\rho)$ as a curve which behaves like a (standard, non--space-filling) $\slekp(\rho)$, but each time that the latter disconnects a region from its target point, the space-filling variant instead fills up this region before continuing on its path.

To define space-filling SLE, we will use the \emph{space-filling ordering}, introduced in \cite{ms2017ig4}. We will describe only the case we need: that of a clockwise space-filling $\slekp(\kk' - 6)$ loop in $\D$ started from $-i$ (with force point at $-i^+$).
Let $h$ be a GFF on $\D$ with boundary conditions $\la'$ to the left of $-i$ and $\la' - 2\chi\pi$ to the right of $-i$.
For each $z \in \D$, let $\eta^+_z$ and $\eta^-_z$ be the (interior) flow lines of $h$ started from $z$ of angle $\pi/2$ and $-\pi/2$, respectively, as defined in \cite{ms2017ig4}. In order for this to make sense, we need to specify how these flow lines bounce off the boundary; we require that each flow line always bounces in a clockwise direction or terminates if this is not possible.
Then, we say that $z_m$ comes before $z_n$ in the ordering, and write $z_m \prec z_n$, if $z_m$ lies in a connected component of $\D \sm (\eta_n^- \cup \eta_n^+)$ whose boundary consists of part of the right side of $\eta_n^-$ and/or part of the left side of $\eta_n^+$, and possibly also part of $\del \D$.

It is proved in \cite{ms2017ig4} (see also \cite{msw2017cleperc} for the space-filling loop case) that for a fixed countable dense set of points $(z_n)$ in $\D$, there a.s.\ exists a unique space-filling loop $\eta'$ starting and ending at $-i$ which hits the points $(z_n)$ in the order specified by the space-filling ordering, which is continuous when parameterised by area, and for which the set of times $\eta\nv(z_n)$ is dense in $[0,\infty)$.
The law of $\eta'$ does not depend on the choice of $(z_n)$. We define the clockwise space-filling $\slekp(\kk'-6)$ loop starting at $-i$ to be this curve $\eta'$.
If we fix a point $z \in \del\D$, and consider $\eta'$ parameterised by half-plane capacity viewed from $z$, the curve $\eta'_z$ we obtain is the counterflow line of $h$ from $-i$ to $z$ and has the law of an $\slekp(\kk-6)$ targeted at $z$.
Due to its parameterisation, once $\eta'_z$ disconnects a region from $z$, it fills it in instantaneously, as opposed to $\eta'$, which fills it in gradually.
That is, up to reparameterisation, $\eta'_z$ corresponds to the curve $\eta'$ restricted to the times when the tip of $\eta'$ is exposed to $z$. This justifies calling referring to this curve as space-filling $\slekp(\kk-6)$, and explains the heuristic view of the space-filling curve mentioned above.

In Lemmas~\ref{lem:sfsle_pocket_interaction_48} and~\ref{lem:sfsle_pocket_interaction_834} we prove that if $\eta'$ is a space-filling $\sleu$ loop starting from $-i$ coupled with a GFF $h$ as above, and if $\GG$ is either a $\clekp$ coupled to the same GFF as defined in Section~\ref{subsec:cle48}, or a $\CLE_\kk$ coupled to $h$ as defined in Section~\ref{subsec:cle834}, with $\kk' = 16/\kk$, then $\eta'$ fills in the connected component (or pocket) of $\D\sm\Uu$ containing $z$ in one go.
That is, as soon as $\eta'$ enters this pocket, it hits every point in the pocket before hitting any point outside the closure of the pocket. This is the key fact used to prove Theorem~\ref{thm:diam}.
Note that the boundary conditions of $h$ needed for all three constructions are same; they are given by $\la'$ to the left of $-i$ and $\la' - 2\chi\pi$ to the right of $-i$.

\begin{lemma}\label{lem:sfsle_pocket_interaction_48}
Let $\eta'$ be a clockwise space-filling $\sleu$ loop in $\D$ starting from $-i$ coupled with a GFF $h$ as above, and let $\GG$ be a $\clekp$ coupled to the same GFF in the manner described in Section~\ref{subsec:cle48}. 
For a fixed point $z \in \D$, let $s_z$ be the connected component (or pocket) of $\D\sm\Uu$ containing $z$. 
Then $\eta'$ a.s.\ fills in $s_z$ one go; that is, as soon as $\eta'$ enters $s_z$, it hits every point in $s_z$ before hitting any point outside $\ov{s_z}$.
\end{lemma}
\begin{proof}
By conformal invariance, it suffices to consider the case that $z = 0$, and we write $s \equiv s_0$. As described in Section~\ref{subsec:cle48}, $s$ can be found by considering the counterflow line $\eta_0'$ of $h$ from $-i$ to $0$, stopped the first time it makes a clockwise loop around $0$. 
Let $w$ be any point not in $\ov s$. The position of $w$ in the space-filling ordering is determined by the flow lines $\eta_w^\pm$ of angle $\pm\pi/2$ of $h$ started from $w$.
By flow line duality \cite[Theorem~1.13]{ms2017ig4}, we know that $\eta_w^+$ (resp. $\eta_w^-$) is the left (resp.\ right) boundary of $\eta'_w$ (recall that this is the counterflow line of $h$ from $-i$ targeted at $w$). We will show that $\eta'_w$ cannot enter $s$. By the definition of our coupling, $\eta'_w$ and $\eta'_0$ agree up until the point at which these curves separate $w$ from $0$. 
Afterwards, the tip $\eta'_w(t)$ is always in the closure of the connected component of $\D\sm\eta'_w([0,t])$ containing $0$, so therefore cannot be in $s$. By the definition of $s$, $\eta'_0$ doesn't enter the interior of $s$ before time $\tau_0$ (the time at which $\eta'_0$ closes the pocket $s$ and determines $s$). Our claim that $\eta'_w$ cannot enter $s$ follows, meaning that its left and right boundaries $\eta_z^\pm$ do not enter $s$ either. It follows from the definition of the space filling ordering that a.s.\ there do not exist (rational) points $u_1, u_2 \in s$ such that $u_1 \prec w \prec u_2$ in the space-filling ordering, and hence that either all (rational) points in $s$ come before $w$ in the ordering, or after it. 
We can use the continuity of $\eta'$ combined with the fact that the times $(\eta')\nv(z_n)$ are dense to extend this result to all $u_1, u_2 \in s$ and all $w \notin \ov s$. This means that a.s.\ that there exists no triple of points $w \notin \ov s$ and $u_1, u_2 \in s$ such that $u_1 \prec w \prec u_2$, meaning exactly that the pocket $s$ is filled in by $\eta'$ in one go, completing the proof.
\end{proof}

\begin{remark}
Note that the above lemma does not hold if we look at \emph{loops} of $\GG$ rather than pockets. The interior of a loop $\ell$ will consist of a countable union of pockets, and it may be the case that $\eta'$ fills in one of these pockets and then enters the interior of another loop before returning to fill in a different pocket in $\ell$.
\end{remark}

\begin{lemma}\label{lem:sfsle_pocket_interaction_834}
Let $\eta'$ be a clockwise space-filling $\sleu$ loop in $\D$ starting from $-i$ coupled with a GFF $h$ as above, and let $\GG$ be a $\CLE_\kk$ coupled to the same GFF in the manner described in Section~\ref{subsec:cle834}. For a fixed $z \in \D$, let $\ell_z$ denote the loop surrounding $z$ (which exists and is simple a.s.).
Then $\eta'$ a.s.\ fills in (the interior of) $\ell_z$ in one go; that is, as soon as $\eta'$ enters $\ell_z$, it hits every point in its interior before hitting any point outside $\ell_z$.
\end{lemma}
\begin{proof}
We will prove that the result holds a.s.\ for the loop $\ell \equiv \ell_0$ of $\GG$ surrounding $0$. The full result follows since we can deduce the analogous result for the loop surrounding any point $w \in \D$ via conformal invariance. 
Recall from Section~\ref{subsec:cle834} that $\ell$ is determined by a radial $\slekk$ curve started from $-i$ and targeted at $0$ which is coupled with $h$ in terms of a collection of chordal $\slekk$ curves $(\eta_k)_{k \in \N}$.
Let $\eta$ denote the concatenation of these curves, which is a radial $\slekk$ from $-i$ to $z$. Let $(\tau_k)$ be the stopping times associated to the curves $(\eta_k)$; that is, $\tau_k$ is the first time that $\eta_k$ disconnects $0$ from $y_k$.
In the following, we will only ever consider $\eta_k$ before time $\tau_k$ and will write $\eta_k$ to mean $\eta_k|_{[0,\tau_k]}$. 
As before, let $D_k$ denote the connected component of $\D \sm (\eta_1 \cup \dots \cup \eta_{k-1})$ containing $0$, so that $\eta_k$ is a curve in $D_k$ from $\eta_{k-1}(\tau_{k-1})$ to a point $y_{k-1}$.
As explained in Section~\ref{subsec:cle834}, for each stopping time $\tau$ of the curve $\eta$, $\eta([0,\tau])$ is a local set of $h$ with known boundary conditions.
We will define our choice of the points $y_k$ explicitly. Let $\ff_k \colon D_k \to \h$ be the (unique) conformal map sending $\eta_{k-1}(\tau_{k-1})$ to $0$ and $0$ to $i$. Then set $y_k = \ff_k\nv(\infty)$.
See Figure~\ref{fig:all_bcs} for a depiction of this setup.

\begin{figure}[hbtp]
    \centering
    \includegraphics[scale=0.9]{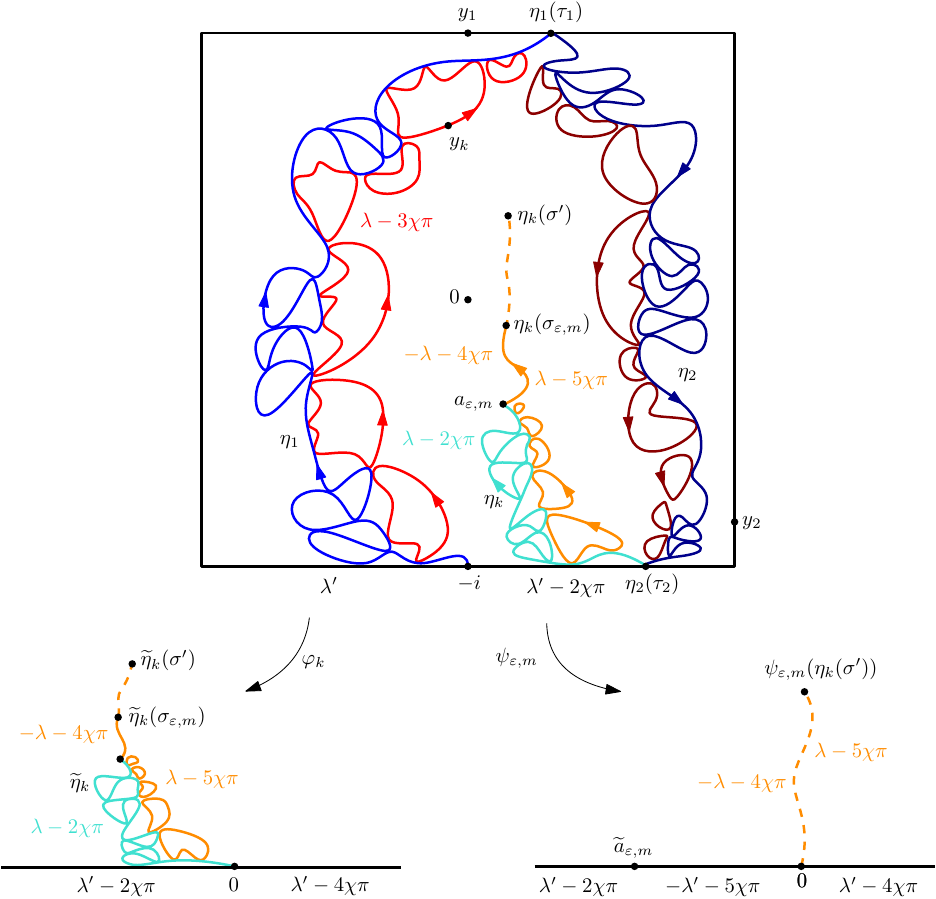}
    \caption{In the top figure, we depict $\eta_1, \cdots, \eta_{k-1}$ and $\eta_k|_{[0, \ss_{\ee,m}]}$ (in the case $k = 3$). Comparing to Figure~\ref{fig:eta_y_bcs}, we can deduce the boundary conditions for the field $h_k = h|_{D_k}$ conditional on $\eta_1, \cdots, \eta_{k-1}$. We can similarly determine the boundary conditions of $h_{\ee,m}$ along $\eta_k([0, \ss_{\ee,m}])$.
    We include the bottom figures to make each of these statements precise. On the left, we depict the boundary conditions of $\wt h_k = h_k \circ \ff_k\nv - \chi\arg(\ff_k\nv)'$ which are given (up to a multiple of $2\pi\chi$ caused by the choice of branch of $\arg$) by $\la' - 2\chi\pi$ on $\R_-$ and by $\la' - 4\chi\pi$ on $\R_+$. 
    Note that the multiple of $2\chi\pi$ for $\wt h_k$ depends on the branch of $\arg$ chosen, but the multiple of $2\chi\pi$ along the curve $\eta_k$ \text{is} uniquely determined. 
    It depends solely on how many times the curves $\eta_1, \dots, \eta_{k-1}$ have wound around $0$.
    Let $a_{\ee,m} = \eta_k(\aal_{\ee,m})$ be the starting point of the excursion loop containing $\eta_k(\ss_{\ee,m})$. More specifically, let $a_{\ee,m}$ denote the prime end at this point which is on the left side of $\eta_k$.
    Let $\wt{a}_{\ee,m}$ be the image of this prime end under $\psi_{\ee,m}$. 
    Then the boundary conditions of $\wt h_{\ee,m} = h_{\ee,m} \circ \psi_{\ee,m} - \chi\arg(\psi_{\ee,m}\nv)'$ are given by $\la' - 2\chi\pi$ on $(-\infty, \wt{a}_{\ee,m})$, by $-\la' - 5\chi\pi$ on $(\wt{a}_{\ee,m}, 0)$ and by $\la' - 4\chi\pi$ on $\R_+$. Again, these boundary conditions are only correct up to a multiple of $2\chi\pi$.
    To determine the boundary conditions of $h_{\ee,m}$ conditional also on $\eta_k([\ss_{\ee,m}, \ss'])$, we just repeat the above description with $\ss'$ in place of $\ss_{\ee,m}$.
    This allows us to identify $\eta_k([\ss_{\ee,m}, \infty))$ as a flow line of $h_{\ee,m}$.
    }
    \label{fig:all_bcs}
\end{figure}

Each curve $\eta_k$ consists of a trunk and excursion loops hanging off the right side of this trunk. Let $L^k = (\ell^k_j)$ denote the collection of such loops drawn by $\eta_k$, and let $L$ be the union of all the $L_k$.
We now explain why $\ell$ is a.s.\ an excursion loop of one of the curves $\eta_k$. 
If, at time $\tau_1$, an excursion loop of $\eta_1$ has been drawn around $0$, then we are done. If this occurs, then $\tau_1$ is necessarily the closing time of this excursion loop.
Otherwise, $0$ is cut off from $y_1$ by $\eta_1$ before $\eta_1$ has drawn $\ell$, and we must continue and look at $\eta_2$.
If none of $\eta_1, \dots, \eta_{k-1}$ have drawn the loop $\ell$ by the times $\tau_1, \dots, \tau_{k-1}$, respectively, it follows from conformal invariance that the probability that $\eta_{k}$ draws $\ell$ before $\tau_{k}$ is constant (and positive) because of our choice of $y_k$. 
It follows that a.s.\ there exists $k < \infty$ such that $\ell$ is drawn by $\eta_k$. That is, $\ell \in L$ a.s.

For each rational point $w \in \D$, let $\eta_w^\pm$ be the (interior) flow lines of $h$ with angle $\pm \pi/2$, respectively. We claim that a.s.\ no flow line $\eta_w^\pm$ started outside any loop $\ell_j^k \in L$ enters the interior of $\ell_j^k$ (in the following, we will use $\ell_j^k$ to denote both the loop and its interior---this should not cause any confusion). The remainder of the proof will be devoted to the proof of this fact.
Assuming the claim, since the collection of flow lines $(\eta_w^\pm)_{w \in \D \cap \Q^2}$ determine the space-filling $\sleu$ $\eta'$, we can conclude by repeating the argument at the end of Lemma~\ref{lem:sfsle_pocket_interaction_48} that every loop $\ell_j^k$ in $L$ (including $\ell$) is a.s.\ filled in in one go by $\eta'$, completing the proof.
It remains therefore to prove the claim.

Fix $k \in \N$. We will prove that for all $j \in \N$ and all rational $w \in \D$ which are not in $\ov{\ell_j^k}$, the curves $\eta_w^\pm$ do not enter $\ell_j^k$ a.s.
Heuristically, the basic idea is to let $\eta_k$ draw a small part of $\ell_j^k$, and then show that the remainder of $\ell_j^k$ is a flow line of the field conditional on $(\eta_j([0, \tau_j]))_{j < k}$ and on the part of $\eta_k$ drawn up to this point. 
However making this precise requires a slightly different argument which needs to consider all loops $(\ell_j^k)_{j \in \N}$ simultaneously.

To this end, fix $\ee > 0$. Recall from Section~\ref{subsubsec:gen_sle} that associated to the $\SLE_\kk(\kk-6)$ curve $\eta_k$ are two processes which we call $W$ and $O$ (these are obtained by mapping $D_k$ to $\h$ and considering the image of $\eta_k$).
Define $X_t = (W_t - O_t)/\sqrt{\kk}$.
Let $\bb_{\ee,0} = 0$ and define for $m \geq 1$ 
\begin{gather*}
    \ss_{\ee, m} = \inf\{t > \bb_{\ee, {m-1}} \colon X_t = \ee\},\\
    \aal_{\ee, m} = \sup\{t < \ss_{\ee, {m}} \colon X_t = 0\},\\
    \bb_{\ee, m} = \inf\{t > \ss_{\ee, {m}} \colon X_t = 0\},
\end{gather*}
Note that $\ss_{\ee,m}$ and $\bb_{\ee,m}$ are stopping times for $\eta_k$ (hence also for $\eta$), but $\aal_{\ee,m}$ is not.
The interval $[\aal_{\ee,m}, \bb_{\ee,m}]$ is the $m$th excursion of $X$ away from $0$ on which $X$ attains a maximum of at least $\ee$, and $\ss_{\ee,m}$ is first time that $X$ reaches height $\ee$ during this interval.

Now, we describe the boundary conditions for the GFF $h$ conditional on certain local sets. 
The set $\eta_1 \cup \cdots \cup \eta_{k-1}$ is a local set for $h$ (recall that $\eta_j$ denotes $\eta_{j}|_{[0, \tau_j]}$, and here we identify this curve with its range). 
Let $h_k$ denote the restriction of $h$ to $D_k$ and define $\wt h_k = h_k \circ \ff_k\nv - \chi\arg(\ff_k\nv)'$ (where $\ff_k \colon D_k \to \h$ is as above). 
Then, conditional on $\eta_1, \cdots, \eta_{k-1}$, the field $h_k$ has the law of a GFF on $D_k$  with boundary conditions such that $\wt h_k$ is a GFF on $\h$ with boundary conditions (up to a multiple of $2\pi\chi$ caused by the choice of branch of $\arg$) given by $\la'$ on $\R_-$ and by $\la' - 2\chi\pi$ on $\R_+$.

Conditional on $\eta_1, \cdots, \eta_{k-1}$, recall that $\eta_k$ is (up to parametrisation) an $\slekk$ on $D_k$ from $\eta_{k-1}(\tau_{k-1})$ to $y_k$, stopped when it disconnects $0$ from $y_k$, which is coupled to $h_k$ as a generalised flow line, as defined in Section~\ref{subsec:cle834}. 
Since $\eta_k$ is coupled to $h_k$ as a generalised flow line, we have that $\eta_k([0, \ss_{\ee,m}])$ is a local set of $h_k$.
Denote by $D_{\ee,m}$ the component of $D_k \sm \eta_k([0, \ss_{\ee,m}])$ containing $0$ and let $h_{\ee,m}$ be the restriction of $h$ to $D_{\ee,m}$. 
The boundary conditions of $h_{\ee,m}$ conditional on $\eta_1, \cdots, \eta_{k-1}$ and $\eta_k([0, \ss_{\ee,m}])$ are described in Figure~\ref{fig:all_bcs}.

Recall that the curve $\eta_k|_{[\aal_{\ee,m}, \bb_{\ee,m}]}$ corresponds to an excursion of the process $X$ and is a.s.\ a simple loop. 
Consider now the remainder of this excursion loop, $\eta_k|_{[\ss_{\ee,m}, \bb_{\ee,m}]}$. If $\ss$ is any stopping time for $\eta_k$ which is a.s.\ in $(\ss_{\ee,m}, \bb_{\ee,m})$ then $\eta_k([0, \ss])$ is a local set of $h_k$ with boundary conditions again as in Figure~\ref{fig:all_bcs}. 
In particular, $\eta_k([\ss_{\ee,m}, \ss])$ is a local set of $h_{\ee,m}$ and has boundary conditions (up to a multiple of $2\chi\pi$) given by $\la - 5\chi\pi$ (resp. $-\la - 4\chi\pi)$ on (a hypothetical) vertical segment on the right (resp.\ left) side of the curve\footnote{The multiple of $2\pi\chi$ depends on the winding of $\eta_1, \dots, \eta_{k-1}$ around $0$, and the winding (in the sense of \cite{ms2016imag1}) of $\eta_k$ on this vertical segment}.
Its true boundary conditions are determined in terms of its winding, as defined in \cite{ms2016imag1}. 
In particular, $\eta_k([\ss_{\ee,m}, \ss])$ has exactly the same boundary conditions as a flow line of $h_{\ee,m}$ of angle $9\chi\pi/2$ (up to multiples of\footnote{This angle is correct only up to possible multiples of $2\pi$. The multiple in question depends in principle on the winding of $\eta_1, \dots, \eta_{k-1}$ and $\eta_{k}([0, \ss_{\ee,m}))$. However, as is explained further in Footnote~\ref{ftn:FL_angle}, without some additional information, the angle of a flow line in a general domain is only defined up to multiples of $2\pi$. One must specify some additional information in order to determine this multiple. In the situation in Figure~\ref{fig:all_bcs}, the boundary conditions agree with those of a flow line of angle $9\chi\pi/2$, but in different situations this value may be $\pi/2 + 2m\pi$ for any $m \in \Z$. In the remainder of this proof, we will assume that we are in the setting where this angle is indeed $9\chi\pi/2$.} $2\pi$), and by uniqueness \cite[Theorem~1.1]{ms2016imag1}, we conclude that $\eta_k|_{[\ss_{\ee,m}, \bb_{\ee,m}]}$ is exactly this flow line. 
Note that by \cite[Theorem~1.1]{ms2016imag1}, we can see from the boundary conditions of $h_{\ee,m}$ that this flow line has the law of an $\SLE_\kk$ from $\eta_k(\ss_{\ee,m})$ to the prime end of $D_{\ee,m}$ corresponding to the start point of the excursion loop which is on the left side of $\eta_k$ (see Figure~\ref{fig:all_bcs}).

Let $\eta_{\ee,m}$ be the flow line of $h_{\ee,m}$ of angle $9\chi\pi/2$ (again, this angle may differ by multiples of $2\pi$, but this will not affect our argument) starting from $\eta_k(\ss_{\ee,m})$. 
For each $w \in \D$ let $\wt\eta^{\,+}_w$ be the (interior) flow line of angle $\pi/2$ of $h_{\ee,m}$ started from $w$. Note that $\wt\eta^{\,+}_w$ is only defined when $w \in D_{\ee,m}$.
By \cite[Proposition~3.7]{ms2017ig4}, $\eta_{\ee,m}$ can a.s.\ be decomposed into segments of interior flow lines of $h_{\ee,m}$ of angle\footnote{Interior flow lines of angle $\pi/2$ and $\pi/2 + 2m\pi$ coincide for $m \in \N$.} $\pi/2$. 
By this, we mean that there exist a collection of rational points $w_n \in D_{\ee,m}$ such that $\eta_{\ee,m}\sm\{\eta_k(\ss_{\ee,m})\}$ is contained in $\bigcup_n \wt\eta^{\,+}_{w_n}$.
Denote by $\eta^+_w$ the flow line of angle $\pi/2$ of $h$ (our original GFF on $\D$) starting from $w \in \D$. 
On the event that a fixed point $w \in \D$ is in $D_{\ee,m}$, then by flow line uniqueness \cite[Theorems~1.1 and~1.2]{ms2017ig4}, this curve is equal to $\wt\eta^{\,+}_{w}$ up until the first time that the curves hit $\del D_{\ee,m}$.
Identifying $\eta_{\ee,m}$ and $\eta_k|_{[\ss_{\ee,m}, \bb_{\ee,m}]}$, it follows that $\eta_k((\ss_{\ee,m}, \bb_{\ee,m}))$ is contained in $\bigcup_n \eta^{+}_{w_n}$.

Let us now conclude. Suppose that $\ell_j^k$ is an excursion loop of $\eta_k$, and suppose it it is traced by $\eta_k|_{[\aal,\bb]}$. Let $x_j^k$ be its start and endpoint, which we refer to as its special point.
Then, there will exist $\ee > 0$ such that $\sup X|_{[\aal, \bb]} > \ee$, so there must exist some $m$ such that $\ss_{\ee,m} \in (\aal, \bb)$ and hence that $\aal = \aal_{\ee,m}$ and $\bb =  \bb_{\ee,m}$. 
Hence, it follows from the above that the part of the loop $\ell_j^k$ traced by $\eta_k|_{[\ss_{\ee,m}, \bb]}$ is contained in a union of flow lines of $h$ of angle $\pi/2$.
But by considering $\ee_n = 1/n$ for all $n \in \N$, we have a.s.\ that $\ss_n := \inf\{t > \aal \colon X_t = 1/n\}$ converges to $\aal$ as $n \to \infty$. 
For each $n$, there exists $m_n \in \N$ such that $\ss_n = \ss_{\ee_n, m_n}$. By arguing as above for all $n$, we see that every segment of $\ell_j^k$, which is at a positive distance from $x_j^k$, must be contained in some interior flow line of $h$ of angle $\pi/2$.

It follows from the flow line interaction rules \cite[Theorem~1.7]{ms2017ig4} that if $u, w \in \D$, then neither $\eta_u^+$ nor $\eta_u^-$ can cross $\eta_w^+$ a.s.
It follows that if $w \in \D$ is not in the closure of $\ell_j^k$, then $\eta_w^+$ (the same argument works for $\eta_w^-$) cannot cross $\ell_j^k$ except possibly at $x_j^k$.
However, suppose that $\eta_w^+$ does cross into $\ell_j^k$ at this special point. Then $\eta_w^+$ must leave $\ell_j^k$ again (since it must terminate somewhere on $\del \D$), and by the above can only leave $\ell_j^k$ through this same special point.
But $\eta_w^+$ can only intersect itself after winding around its startpoint \cite{ms2017ig4}, and $w$ is not in the interior of $\ell_j^k$. Therefore this could not have occurred, leading to a contradiction. 
We conclude therefore that $\eta_w^+$ (and similarly $\eta_w^-$) a.s.\ does not enter the interior of $\ell_j^k$.
This completes the proof of the claim, and also the lemma.
\end{proof}

\section{An area-filling property of space-filling SLE}\label{sec:area_filling}
In this section we will prove the following result. 
\begin{proposition}\label{prop:area_filling}
Fix $\kk' \in (4,8)$ and let $\eta'$ be a clockwise space-filling $\SLE_{\kk'}(\kk' - 6)$ loop in $\D$ starting from $-i$. For $\ee > 0$ and $\xi > 1$ define $E_{\ee, \xi}$ to be the event that for all $\dd \in (0,\ee]$ and for all $0 \leq a < b < \infty$ with $\diam(\eta'([a,b])) \geq \dd$, the set $\eta'([a,b])$ contains a ball of radius $\dd^\xi$. Then, for a fixed $\xi > 1$, $\p[E_{\ee, \xi}]$ converges to $1$ faster than any power of $\ee$ as $\ee \to 0$. That is, for any $\xi > 1$ and $n \in \N$, there exists a constant $A_{n,\xi} > 0$ such that for all $\ee > 0$,
\[\p[E_{\ee, \xi}^c] \leq A_{n,\xi}\,\ee^n.\]
\end{proposition}
This proposition is an extension of \cite[Proposition~A.2]{miller2021tightness}, which states that for each $\xi > 1$ there almost surely exists $\ee > 0$ such that $E_{\ee, \xi}$ holds. 
Proposition~A.2 in \cite{miller2017dimension} is, in turn, an adaptation of \cite[Proposition~3.4 and Remark~3.9]{ghm2020almost}, which we state here in the form given in \cite[Lemma~A.3]{miller2021tightness}.

\begin{lemma}[{\cite[Proposition~3.4 and Remark~3.9]{ghm2020almost}}]\label{lem:ghm_lemma}
Fix $\kk' \in (4,8)$ and suppose that $\eta_w'$ is a space-filling $\SLE_{\kk'}$ from $\infty$ to $\infty$ in $\C$. For $\xi > 1, R > 0$ and $\ee > 0$ let $E^w_{\ee, \xi, R}$ be the event that the following is true. For all $\dd \in (0, \ee]$ and $0 \leq a < b < \infty$ with $\eta_w'([a,b]) \subseteq B(0,R)$ and $\diam(\eta_w'([a,b])) \geq \dd$, the set $\eta_w'([a,b])$ contains a ball of radius $\dd^\xi$. Then, for fixed $\xi > 1, R > 0$, $\p[E^w_{\ee, \xi, R}]$ converges to $1$ faster than any power of $\ee$ as $\ee \to 0$. That is, for each $\xi > 1, R > 0$ and $n \in \N$ there exists $C_{n, \xi, R} > 0$ such that for all $\ee > 0$,
\[\p[(E^w_{\ee, \xi, R})^c] \leq C_{n, \xi, R}\,\ee^n.\]
\end{lemma}

That is, \cite[Proposition~A.2]{miller2021tightness} adapts Lemma~\ref{lem:ghm_lemma} from the case of whole-plane space-filling SLE to the case of a space-filling SLE loop in the disk, but leads only to an almost sure result. We wish to also transfer the faster than polynomial decay rate of Lemma~\ref{lem:ghm_lemma} to the space-filling loop case.
To do so, our proof will broadly follow the proof of \cite[Proposition~A.2]{miller2021tightness}, but getting the required decay rate for $\p[E_{\ee, \xi}^c]$ will require some extra care.

Before proving the proposition, we first prove a lemma on the absolute continuity of the laws of the whole-plane and zero-boundary GFFs.
Our proof is based on \cite[Lemma~4.1, Lemma~4.4]{mq2020geodesics}, but we include the full argument here.

\begin{lemma}\label{lem:gff_ac}
Let $L > 0$ and $r \in (0,1)$. Let $h^0_\D$ be a zero boundary GFF on $\D$, let $f$ be a harmonic function on $\D$ with $\nn{f}_\infty \leq L$ and define $h = h^0_\D + f$. Let $h_w$ be a whole-plane GFF with additive constant fixed so that its average on $\del \D$ is uniform in $[0, 2\pi\chi$) (as defined in Section~\ref{subsec:gff}). Then the laws of $h$ and $h_w$ restricted to $B(0,r)$ are mutually absolutely continuous. Furthermore, if we denote by $\CZ$ the Radon--Nikodym derivative of the law of $h|_{B(0,r)}$ with respect to the law of $h_w|_{B(0,r)}$, then there exists $p \in (1,\infty)$ such that $\E[\CZ^p] < \infty$.
\end{lemma}

\begin{proof}
Choose $r < r_1 < r_2 < r_3 < R < 1$. Using the Markov property of the GFF we can couple $h$ and $h_w$ so that $h = \fh + h^0$ and $h_w = \fh_w + h^0$, where $h^0 \equiv h_R^0$ is a zero-boundary GFF on $B(0,R)$ and $\fh$ and $\fh_w$ are distributions which are harmonic functions on $B(0,R)$, each independent of $h^0$. 
We assume the coupling is such that $\fh$ and $\fh_w$ are independent (note that we have also assumed that the projections of $h$ and $h_w$ onto $H(B(0,R))$ (as defined in \cite{ms2016imag1}) are the same).
Let $\phi$ be a smooth function which is equal to $1$ on $B(0,r)$ and $0$ outside $B(0,r_1)$, and set $g = (\fh - \fh_w)\phi$. Writing $B \equiv B(0,r)$ for convenience, note that $h|_B = h_w|_B + g$.
By e.g.\ \cite[Proposition~1.51]{bp_gff}, the Radon--Nikodym derivative of the conditional law of $h^0 + g$ given $(\fh, \fh_w)$ (equivalently given $g$) with respect to the law of $h^0$ is given by $\CZ_1 = \exp((h^0, g)_\nabla - \frac12(g,g)_\nabla)$.
It follows that the Radon--Nikodym derivative for the joint law of $(h|_B, \fh, \fh_w)$ with respect to the joint law of $(h_w|_B, \fh, \fh_w)$ is also given by $\CZ_1$. Then, the Radon--Nikodym derivative of $h|_B$ with respect to $h_w|_B$ is $\cZ = \E[\cZ_1 \giv h_w|_B] = \E[\exp((h^0, g)_\nabla - \frac12(g,g)_\nabla) \giv h_w|_B]$, where $h, h_w$ are coupled as above.

By Jensen's inequality,
\begin{align*}
\E[\cZ^p] &= \E\left[\E\big[\exp((h^0, g)_\nabla - \tfrac12(g,g)_\nabla) \big{|} h_w|_B\big]^p\right]\\
&\leq \E\left[(\exp((h^0, g)_\nabla - \tfrac12(g,g)_\nabla))^p\right]\\
&= \E\big[\exp(-\tfrac{p}{2}(g,g)_\nabla)\, \E[\exp(p(h^0, g)_\nabla)\giv g]\big]\\
&= \E\left[\exp\left(\frac{p^2 - p}{2} \nn{g}_\nabla^2 \right)\right],
\end{align*}
where in the last equality we have used the fact that conditional on $g$, $(h^0,g)$ is a centred Gaussian random variable with variance $\nn{g}^2_\nabla$.

By the product rule, there exists a constant $C_1$ (depending on our choice of $\phi$) such that
\[\nn{g}_\nabla^2 \leq C_1\left(\nn{\fh}_{\infty, B(0, r_1)}^2 + \nn{\fh_w}_{\infty, B(0, r_1)}^2 + \nn{\fh}_{\nabla, B(0, r_1)}^2 + \nn{\fh_w}_{\nabla, B(0, r_1)}^2 \right).\]
Here, $\nn{\cdot}_{\cdot, B(0,r_1)}$ denotes that the norm is to be taken on $B(0,r_1)$, which can be done since $\phi$ is supported on $B(0,r_1)$. 
By a standard harmonic function estimate (see e.g.\ \cite{evans_pdes}) there exists a universal constant $c > 0$ such that if $S > 0$ and $u$ is harmonic on $B(0, S)$ then for $w \in B(0, S)$ we have
\[\n{\nabla u(w)} \leq \frac{c}{\dist(w, \del B(0, S))} \sup_{v \in B(0,S)}\n{u(v) - u(0)}.\]
Applying this to $\fh, \fh_w$ with $S = r_2$ (so that $\dist(w, \del B(0, r_2))$ is bounded below by a positive constant for $w \in B(0,r_1)$) we find there exists a constant $C_2$ (depending on $r_1, r_2$ and $\phi$) such that (note that we are now computing the norms on a larger set)
\[\nn{g}_\nabla^2 \leq C_2\left(\nn{\fh}_{\infty, B(0, r_2)}^2 + \nn{\fh_w}_{\infty, B(0, r_2)}^2 \right).\]
We conclude that 
\begin{equation}\label{eq:rn_bound}
\E[\cZ^p] \leq \E\left[\exp\left(\frac{p^2 - p}{2} \nn{g}_\nabla^2 \right)\right] \leq \E\left[\exp\left(\frac{p^2 - p}{2} C_2\left(\nn{\fh}_{\infty, B(0, r_2)}^2 + \nn{\fh_w}_{\infty, B(0, r_2)}^2 \right)\right)\right].
\end{equation}

Next, we argue as in \cite[Lemma~4.4]{mq2020geodesics} to show that there exists $\la > 0$ such that 
\[\E[\exp(\la \nn{\fh}_{\infty,B(0,r_2)}^2)] < \infty,\]
and the same for $\fh_w$ in place of $\fh$. To do so, let $\fp(\cdot, \cdot)$ denote the Poisson kernel on $B(0, r_3)$. Note that for $z \in B(0,r_2), w \in B(0, r_3)$, the kernel $\fp(z,w)$ is bounded by some constant $C_3$ (depending on $r_2, r_3$). So we have, for all $z \in B(0,r_2)$.
\[ \n{\fh(z)} = \n{\int_{\del B(0,r_3)} \fp(z,w)\fh(w)dw} \leq C_3\int_{\del B(0,r_3)}\n{\fh(w)}dw.\]
Hence the same bound holds for $\nn{\fh}_{\infty, B(0,r_2)}$. By Jensen's inequality we have
\[\exp\left(\la\nn{\fh}^2_{\infty, B(0,r_2)}\right) \leq \int_{\del B(0,r_3)} \exp\left(\la C_3^2 \n{\fh(w)}^2\right)dw,\]
and hence that
\begin{equation}\label{eq:poisson_bound}
\E\left[\exp\left(\la\nn{\fh}^2_{\infty, B(0,r_2)}\right)\right] \leq \int_{\del B(0,r_3)} \E\left[\exp\left(\la C_3^2 \n{\fh(w)}^2\right)\right]dw.
\end{equation}
But for $w \in \del B(0, r_3)$, $\fh(w)$ is normal random variable with mean determined by $f$  and variance which depends on $r_3$ and $R$, but is the same for all $w$. It follows that there exists $\la > 0$ small enough that this expectation is finite. 
The same argument as above shows that \eqref{eq:poisson_bound} also holds for $\fh_w$ in place of $\fh$.
We have normalised $h_w$ so that its average on $\del \D$ is uniform in $[0, 2\pi\chi)$, but if we let $\ov{h}_w$ be a whole-plane GFF normalised so that this average is $0$, we have $h_w = \ov{h}_w + U$, where $U$ is a uniform random variable on $[0, 2\pi\chi)$. Therefore it follows that for $v \in \del B(0,r_3)$, $\fh_w(v)$ is the sum of a normal random variable (with fixed variance) and a uniform random variable, and that the integral is finite also in this case for small enough $\la > 0$.
To conclude, since $p^2 - p$ can be made arbitrarily close to $0$ for $p > 1$, it follows from \eqref{eq:rn_bound} that for $p > 1$ small enough, $\E[\cZ^p] < \infty$, as claimed.
\end{proof}

\begin{proof}[Proof of Proposition~\ref{prop:area_filling}]
\textit{Setup and overview.}
We will first give an overview of our proof. In the first step, we will use the absolute continuity of the zero boundary GFF on $\D$ with respect to a whole-plane GFF, when each is restricted to $B(0,1/2)$. This will allow us to convert the area-filling property of whole-plane space-filling SLE stated in Lemma~\ref{lem:ghm_lemma} to a similar area-filling property for the segments of a space-filling SLE loop $\eta'$ which are contained in $B(0,1/2)$. In Step 2, we will use conformal invariance properties of $\eta'$ and conformal mapping distortion estimates to prove an area-filling property for segments of $\eta'$ in balls $B(z, r)$, where $B(z, 2r) \subseteq \D$. 
In Step 3 we will use a union bound to define an event $H_{\ee, \xi, n}$ on which $\eta'$ satisfies an area-filling property in balls of all scales simultaneously. Intuitively, this means we have control on the area-filling behaviour of segments of $\eta'$ which have small diameter relative to their distance to the boundary (we do not make this statement precise in the proof, but it is a useful way to view what we know at this stage). 
It remains to consider the behaviour of segments of $\eta'$ which are close to the boundary, relative to their diameter. In the fourth step, we show that it is unlikely that $\eta'$ travels a large distance along $\del \D$ without moving away from the boundary. In Step 5, we show how Steps 3 and 4 allow us to understand the area-filling behaviour of all segments of $\eta'$ simultaneously. The sixth step is a short conclusion.

We first define the following event. For $z \in \D, r > 0$ such that $B(z, r) \subset \D$, let $G_{\ee, z, r, \xi}$ be the event that for all $\dd \in (0, \ee]$ and all $0 \leq a < b < \infty$ with $\eta'([a,b]) \subseteq B(z,r)$ and $\diam(\eta'([a,b])) \geq \dd$, the set $\eta'([a,b])$ contains a ball of radius $\dd^\xi$.

\textit{Step 1. An area-filling property on $B(0,1/2)$.}
Let $h$ be a GFF on $\D$ with boundary conditions given by $\la'$ to the left of $-i$ and $\la' - 2\chi\pi$ to the right, and suppose that $\eta'$ is coupled to $h$ as described in Section~\ref{sec:sfsle}.
The first step of the proof is to use the absolute continuity of $h|_{B(0, 1/2)}$ with respect to a whole-plane GFF to convert Lemma~\ref{lem:ghm_lemma} to a result about space-filling SLE on the disk, at least when it is far away from the boundary. 
To this end, let $h_w$ be a whole-plane GFF, normalised so that its average on $\del \D$ is uniform in $[0, 2\pi\chi)$, and let $\eta'_w$ be a whole-plane $\SLE_{\kk'}$ coupled to it.
Then, fixing $R = 1/2$ in Lemma~\ref{lem:ghm_lemma} (we continue to write $R$ for notational reasons), we have that for a fixed choice of $\xi, n$, there exists $C_{n,\xi,R} > 0$ such that $\p[(E^w_{\ee, \xi, R})^c] \leq C_{\ee, \xi, R}\,\ee^n$ for all $\ee > 0$.
By the definition of the coupling of $\eta_w'$ to $h_w$, $\eta'_w$ is a measurable function of the field, and the event $E^w_{\ee, \xi, R}$ can be seen to be measurable with respect to $h_w|_{B(0, 1/2)}$. Similarly, $\eta'$ is a measurable function of $h$, and $G_\ee$ is measurable with respect to $h|_{B(0, 1/2)}$. 
In fact, the measurable functions $h|_{B(0, 1/2)} \mapsto \one_{G_\ee}$ and $h_w|_{B(0, 1/2)} \mapsto \one_{E^w_{\ee, \xi, R}}$ coincide,\footnote{What we mean here is the following. For $z \in \D$, the flow line $\eta_z$ (resp.\ $\eta^w_z$) of $h$ (resp.\ $h^w$) started from $z$ and stopped when it first hits time it leaves $B(0,1/2)$ is a measurable function of $h|_{B(0,1/2)}$ (resp.\ $h^w|_{B(0,1/2)}$). 
If we denote these two functions by $H$ and $H^w$, which map from the set of distributions on $B(0,1/2)$ to some space of curves, then the maps $H$ and $H^w$ must coincide almost everywhere (with respect to the law of $h|_{B(0,1/2)}$). This follows from the definition of the flow line coupling in \cite{ms2017ig4}.
The events $G_\ee$ and $E^w_{\ee,\xi, R}$ depend on $h|_{B(0,1/2)}$ and $h^w|_{B(0,1/2)}$ respectively through collections of flow lines (of angle $\pm\pi/2$) $(\eta_z)$ and $(\eta_z^w)$ stopped when they first exit $B(0,1/2)$, where $z$ ranges over $\D \cap \Q^2$. It follows then that if $h|_{B(0,1/2)} = h^w|_{B(0,1/2)}$, then $G_\ee$ holds if and only if $E^w_{\ee,\xi, R}$ holds (up to a probability $0$ events), meaning that the maps $h|_{B(0, 1/2)} \mapsto \one_{G_\ee}$ and $h_w|_{B(0, 1/2)} \mapsto \one_{E^w_{\ee, \xi, R}}$ coincide (up to a probability $0$ event---we can then modify one of the maps on this measure $0$ event so that they coincide everywhere).} so if we let $\CZ$ denote the Radon--Nikodym derivative of the law of $h|_{B(0, 1/2)}$ with respect to $h_w|_{B(0, 1/2)}$, then $\p[G_\ee] = \E[\CZ \one_{E^w_{\ee, \xi, R}}]$. 
By Lemma~\ref{lem:gff_ac} there exists $p > 1$ such that $\E[\CZ^p] < D_1$, where $D_1$ depends only on $\kk'$ (via the boundary conditions of $h$ on $\D$, which in this case depend only on $\kk'$). Therefore, by H\"older's inequality, with $1/p + 1/q = 1$,
\[\p[G_\ee^c] = \E[\CZ \one_{(E^w_{\ee, \xi, R})^c}] \leq \E[\CZ^p]^{1/p}\,\p[(E^w_{\ee, \xi, R})^c]^{1/q} \leq (D_1^{1/p} C_{n,\xi, R}^{1/q})\,\ee^{n/q}.\]
It follows that for each $n \in \N, \xi > 1$, there exists a constant $C_{n,\xi} > 0$ such that for all $\ee > 0$,
\begin{equation}
    \p[G_{\ee, 0, 1/2, \xi}^{c}] \leq C_{n,\xi}\ee^n.
\end{equation}

\textit{Step 2. Extending this bound to any ball away from the boundary.}
Fix $z \in \D, r > 0$ such that $B(z, 2r) \subseteq \D$ and let $\ff\colon \D \to \D$ be the conformal map sending $z$ to $0$ and fixing $-i$.
Using an explicit expression for $\ff$, we can compute $\n{\ff'(z)} = 1/(1-\n{z})^2$ and show that if $w \in B(z,r)$ then $\n{\ff(w)} \leq 1/2$, meaning that $\ff(B(z,r)) \subseteq B(0,1/2)$. For $w \in B(z,r)$, using \cite[Theorem~3.21]{lawler2008conformally} on $\psi := \ff\nv$ and that $\n{\ff(w)} \leq 1/2$, we get that there exist $d_1 = 1/12$ and $d_2 = 27/4$ such that $d_1 \n{\ff'(z)} \leq \n{\ff'(w)} \leq d_2\n{\ff'(z)}$. Similarly, for $v_1, v_2 \in B(0, 1/2)$ we can compute that
\begin{equation}\label{eq:distortion_bounds_1}
    \frac14(1-\n{z}^2)\n{v_2 - v_1} \leq \n{\psi(v_2) - \psi(v_1)} \leq 4(1-\n{z}^2)\n{v_2 - v_1}.
\end{equation}
Since $\ff = \psi\nv$, we correspondingly have for $u_1, u_2 \in B(z,r)$, where also using that $1 - \n{z}^2 \geq r$,
\begin{equation}\label{eq:distortion_bounds_2}
    \frac{1}{4(1-\n{z}^2)}\n{u_2 - u_1} \leq \n{\ff(u_2) - \ff(u_1)} \leq \frac{4}{1-\n{z}^2}\n{u_2 - u_1} \leq \frac4r \n{u_2 - u_1}.
\end{equation}

By the conformal invariance of the curve $\eta'$, its image $\eta'_z := \psi(\eta')$ has the same law as $\eta'$ (up to time reparameterisation; for this argument we will parameterise $\eta'_z$ as $\eta'_z(t) = \psi(\eta'(t))$).

Fix $\xi > 1$ and choose $\ee > 0$ small enough that $\ee^{\xi - 1} < 1/16$. Suppose that the event $G_{\ee, 0, 1/2, \xi}$ holds for the curve $\eta'$. We will show that $G_{\ee', z, r, \xi_1}$ holds for the curve $\eta'_z$ where $\xi_1 = 2\xi - 1$ and $\ee' = r\ee/4$.
Indeed, suppose $0 \leq a < b < \infty, \eta'_z([a,b]) \subseteq B(z,r)$ and $\diam(\eta'_z([a,b])) = \dd \in (0, \ee']$. Using \eqref{eq:distortion_bounds_2} and the fact that $\n{\ff'(z)} = 1/(1-\n{z})^2$, we have 
\[\frac{1}{4} \n{\ff'(z)}\dd \leq \diam(\eta'[a,b]) \leq \frac{4\dd}{r}.\]
Since $\dd \in (0, \ee']$ we have $\diam(\eta'[a,b]) \leq 4\ee'/r = \ee$, and since $G_{\ee, 0, 1/2, \xi}$ holds, $\eta'([a,b])$ must contain a ball of radius $(\dd\n{\ff'(z)}/4)^{\xi}$. Upon mapping back with $\psi$ and using \eqref{eq:distortion_bounds_1}, we see that $\eta'_z$ contains a ball of radius at least
\[\frac{1}{4}\n{\ff'(z)}\nv(\dd\n{\ff'(z)}/4)^{\xi} = \frac{1}{4^{1+\xi}}\n{\ff'(z)}^{\xi-1}\dd^{\xi} \geq \frac{1}{16}\dd^{\xi} \geq \dd^{\xi + (\xi - 1)} = \dd^{\xi_1},\]
Here we have used that $\xi > 1, |\ff'(z)| > 1$ and that $\dd \leq \ee' < \ee$ so $\dd^{\xi - 1} < 1/16$.
Therefore the event $G_{\ee', z, r, \xi_1}$ holds.
In conclusion, for $\ee' = r\ee/4$ and $\xi_1 = 2\xi - 1$, since $\eta'$ and $\eta'_z$ have the same law, $\p[G_{\ee', z, r, \xi_1}] \geq \p[G_{\ee, 0, 1/2, \xi}]$, meaning that for all $n, \xi$, and for all $\ee > 0$ small enough that $\ee^{\xi-1} < 1/16$,
\[\p[G_{\ee', z, r, \xi_1}^c] \leq C_{n,\xi} \ee^n.\]

\textit{Step 3. Considering the area-filling properties of $\eta'$ on many such balls simultaneously.}
Let $S_m = 2^{-m}\Z^{2} \cap B(0, 1 - 2^{-m+1})$ and note that $\n{S_m} \leq 4\cdot2^{2m}$. Then for $z \in S_m$ and $r = 2^{-m}$, we have $B(z, 2r) \subseteq \D$. Set $\ee_{m,n} = \tfrac14 2^{-m}2^{-3m/n}\ee$, so that $\ee_{m,n} = \frac14 r \wt\ee_{m,n}$ for $\wt\ee_{m,n} = 2^{-3m/n}\ee$, meaning that by Step 2 (as long as $\ee^{\xi - 1} < 1/16$) we have $\p[G_{\ee_{m,n},z,2^{-m}, \xi_1}] \geq \p[G_{2^{-3m/n}\ee, 0, 1/2, \xi}]$.
Define, for $\xi > 1$ and $\ee > 0$,
\[H_{\ee, \xi, n} = \bigcap_{m \geq 1} \bigcap_{z \in S_m} G_{\ee_{m,n}, z, 2^{-m}, \xi}.\]
When $\ee^{\xi - 1} < 1/16$, by Step 2 and a union bound we have (note the use of $\xi_1 = 2\xi - 1$ here)
\begin{equation}
    \p[H_{\ee, \xi_1, n}^c] \leq \sum_{m \geq 1}2^{2m+2} \p[G^c_{2^{-3m/n}\ee, 0, 1/2, \xi}] \leq \sum_{m \geq 1} 2^{2m+2} C_{n,\xi} (2^{-3m/n}\ee)^n = 4C_{n,\xi}\ee^n.
\end{equation}
Finally, since $\xi_1 = 2\xi - 1$ we can do the above for any choice of $\xi_1 > 1$, so replacing $\xi_1$ by $\xi$ in our notation it follows that for any choice of $\xi > 1$ and $n \in \N$, there exists a constant $D_{n,\xi} > 0$ such that for all $\ee > 0$,
\begin{equation}
    \p[H_{\ee, \xi,n}^c] \leq D_{n,\xi}\,\ee^n.
\end{equation}

\begin{figure}[t]
    \centering
    \includegraphics[scale=1]{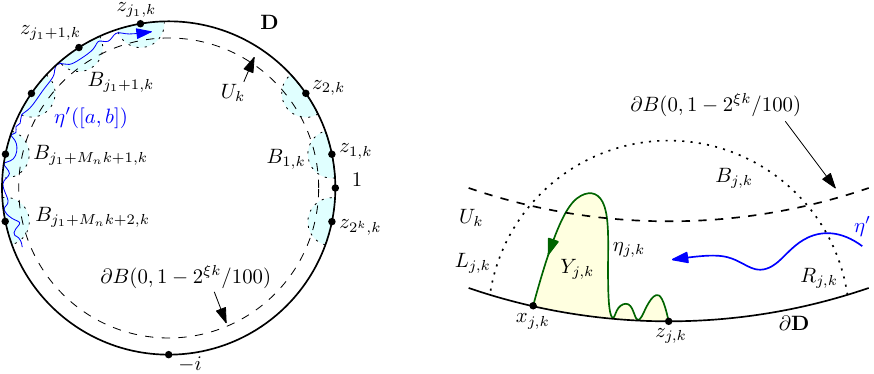}
    \caption{
    On the left we depict the points $z_{j,k} = \exp(2\pi i(j-\tfrac12)/2^k)$, the balls $B_{j,k} = B(z_{j,k}, 2^{-\xi k}/10)$ on which the events $F_{j,k}$ are defined, and the annulus $U_k := \D \sm B(0, 1 - 2^{-\xi k}/100)$, which are defined in Step 4 of the proof of Proposition~\ref{prop:area_filling}. On the right, we depict the event $F_{j,k}$.
    Notice that the distance between two consecutive points $z_{j,k}$ and $z_{j+1,k}$ is bounded above by $(2\pi)2^{-k}$.
    Therefore it follows that if $\diam(\eta'([a,b])) \geq C_n'k2^{-k} = 2\pi(M_n k + 4)$ and $\eta'([a,b]) \subset U_k$, then there must exist points $z_{j_1,k}$ and $z_{j_1+M_n k+2,k}$ (with indices taken modulo $2^k$) such that $\eta'([a,b])$ makes a crossing (in either the clockwise or counterclockwise direction) of the annular sector of the annulus $\D \sm B(0, 1 - 2^{-\xi k}/100)$ between $z_{j_1,k}$ and $z_{j_1+M_n k+2,k}$.
    This ensures that $\eta'([a,b])$ must make crossings (in either the clockwise or counterclockwise direction) of each of the sections of the annulus $U_k$ which are part of the $M_n$ consecutive balls $B_{j_1 + 1, k} \cap U_k, \dots, B_{j_1 + M_n k, k} \cap U_k$. 
    Since we have assumed that $A_{n,\xi, k}$ holds, it follows that there exists a point $z_{j,k}$ such that $F_{j,k}$ holds and $\eta'([a,b])$ makes a crossing of $B_{j,k} \cap U_k$. 
    Here, by a crossing of $B_{j,k} \cap U_k$, we mean that $\eta'$ makes a crossing from the segment of $\del (B_{j,k} \cap U_k)$ marked by $R_{j,k}$ in the right-hand figure to the segment marked $L_{j,k}$ (or from $L_{j,k}$ to $R_{j,k}$).
    On the event that $F_{j,k}$ occurs, to do so $\eta'$ must enter $Y_{j,k}$.
    But, as we argue in Step 4, $\eta'$ must fill in $Y_{j,k}$ in one go, forcing it to leave $U_k$ and causing a contradiction. (We have depicted $\eta'$ as a simple (non--space-filling) curve for clarity.)}
    \label{fig:fjk_events}
\end{figure}

\textit{Step 4. Showing $\eta'$ does not travel along the boundary for too long.}
Here, we mainly follow the proof of \cite[Proposition~A.2]{miller2021tightness} but make some small changes. 
Fix $\xi > 1$.
For each integer $k \geq 2$ let\footnote{We subtract the factor of $1/2$ from $j$ here so that none of our points $z_{j,k}$ are equal to $-i$, where $\eta'$ begins and ends, and where the boundary conditions of $h$ change.} $z_{j,k} = \exp(2\pi i(j-\tfrac12)/2^k)$ for $1 \leq j \leq 2^k$. Let $\eta_{j,k}$ be the flow line of $h$ of angle\footnote{\label{ftn:FL_angle}More precisely, if we consider a conformal map $\psi_{j,k} \colon \D \to \h$ sending $z_{j,k}$ to $0$, with branch of argument chosen such that $h_{j,k} = h \circ \psi_{j,k}\nv - \chi\arg(\psi_{j,k}\nv)'$ has boundary conditions given by $\la'$ on $\R$, then $\psi_{j,k}\nv(\eta_{j,k})$ is a flow line of angle $-\pi/2$ of $h_{j,k}$. This means that when $z_{j,k}$ is to the right of $-i$, where the boundary conditions of $h$ on $\del \D$ are close to $\la' - 2\chi\pi$, we need to account for this difference of $2\chi\pi$ by choosing the branch of $\arg$ accordingly. This leads to the flow line $\eta_{j,k}$ seeming like it has angle $3\pi/2$ rather than $-\pi/2$ in this case. This fundamentally stems from the fact that when we say a (chordal) flow line has angle $\tht$, this implicitly assumes a fixed choice of branch of argument (there is usually a single natural choice), and such a flow line could equally be said to have angle $\tht + 2\pi$, where the choice of branch of $\arg$ is now different. In certain cases like this one, however, there is no natural choice of branch of $\arg$, and we need to explicitly specify which branch we choose.} 
$-\pi/2$ started at $z_{j,k}$, targeted at $-i$, and stopped when it leaves $B_{j,k}:= B(z_{j,k}, 2^{-\xi k}/10)$. Let $F_{j,k}$ be the event that $\eta_{j,k}$ enters $B(0, 1 - 2^{-\xi k}/100)$ and afterwards hits $\del \D$ (at a point we call $x_{j,k}$) before leaving $B_{j,k}$
(see Figure~\ref{fig:fjk_events}).
Then, arguing as in \cite{miller2021tightness} there exists $p > 0$ (not depending on $k$) such that the collection of events $F_{j,k}$ is stochastically dominated by $2^k$ i.i.d.\ Bernoulli random variables with success probability $p$.
Therefore, for each $n \in \N$, there exists $M_n > 0$ such that the probability that there exists a consecutive sequence of length $M_nk$ of the $F_{j,k}$ (along $\del \D$) which all fail is at most $2^{-nk}$. 
Let $A_{n,\xi,k}$ be the event that no sequence of failures of length $M_nk$ exists for a fixed $k$, and let $B_{n,\xi, k} = \bigcap_{k' \geq k} A_{n,\xi, k'}$. It follows that $\p[B_{n,\xi, k}^c] \leq 2\cdot 2^{-nk}$.

Set $C_n' = 2\pi(M_n + 4)$. We claim that if $A_{n,\xi,k}$ occurs, and $\diam(\eta'([a,b])) \geq C'_{n}k2^{-k}$ for some $a < b \in \R$, then $\eta'([a,b])$ must leave the annulus $U_k := \D \sm B(0, 1 - 2^{-\xi k}/100)$.
To prove the claim, suppose that $A_{n,\xi,k}$ occurs, $\diam(\eta'([a,b])) \geq C'_{n}k2^{-k}$ and $\eta'([a,b])$ stays in $\D \sm B(0, 1 - 2^{-\xi k}/100)$. 
In Figure~\ref{fig:fjk_events} we explain that because $\diam(\eta'([a,b])) \geq C_n'k2^{-k}$ and since $A_{n,\xi,k}$ holds, there must exist a point $z_{j,k}$ such that $F_{j,k}$ holds and $\eta'([a,b])$ makes a crossing (from its clockwise boundary to its counterclockwise boundary or vice versa) of $U_k \cap B_{j,k}$.
If $F_{j,k}$ holds, let $Y_{j,k}$ be the region of $\D$ bounded between $\del \D$ and the segment of $\eta_{j,k}$ up until it hits $x_{j,k}$. 
Recall that $\eta_{j,k}$ is targeted at $-i$. Using the boundary conditions of $h$ one can compute that $\eta_{j,k}$ is an $\SLE_\kk(-\kk/2; \kk/2 - 2)$ from $z_{j,k}$ to $-i$ in $\D$ and therefore does not intersect the counterclockwise arc from $z_{j,k}$ to $-i$ a.s.\ by \cite[Lemma~2.1]{mw2017intersections} (we refer to \cite{ms2016imag1} for the definition of SLE with multiple force points). 
In particular, this means that $\eta_{j,k}$ always bounces off $\del \D$ in the clockwise direction.
It follows from the flow line interaction rules \cite[Theorem~1.7]{ms2017ig4} that the flow lines $\eta_w^\pm$ used to define the space-filling ordering a.s.\ do not cross $\eta_{j,k}$ and hence, arguing as in Lemmas~\ref{lem:sfsle_pocket_interaction_48} and~\ref{lem:sfsle_pocket_interaction_834}, we see that $\eta'$ a.s.\ fills in $Y_{j,k}$ in one go.
For $\eta'$ to make a crossing of $U_k \cap B_{j,k}$, it must enter $Y_{j,k}$ (since $F_{j,k}$ holds), and therefore must fill in all of $Y_{j,k}$ before leaving $B_{j,k}$. Since $F_{j,k}$ holds, $Y_{j,k}$ contains points in $B(0, 1 - 2^{-\xi k}/100)$, which is a contradiction, thus proving the claim.

\textit{Step 5. Determining area-filling properties for a general segment of $\eta'$.}
Fix $\xi > 1$ and $n \in \N$. Choose $K = K(\xi) \in \N$ large enough that $2C'_{n}k2^{-k} < 2^{-k/\xi}$ for all $k \geq K$. Assume $\ee \in (0, \ee_0)$, where $\ee_0 = \ee_0(\xi)$ depends only on $\xi$ (directly and through $K(\xi)$) and is small enough that $\ee_0 < C_n'K2^{-K}$ and $\ee_0^{\xi - 1} < 2^{-42\xi-4\xi^2}$.
Set $k_0 = \lceil \xi\log_2 (1/\ee)\rceil$. Note that $\ee < 2^{-K/\xi}$, meaning that $K < \xi\log_2(1/\ee)$, so $k_0 \geq K$.
Suppose that $H_{\ee,\xi,n}$ and $B_{n,\xi,k_0}$ both hold. We will show that for any choice of $q \in \N$, the event $E_{\ee_2, \xi_2}$ holds, where 
\[\ee_2 = \ee_2(\ee, q) = \ee^q, \qquad \xi_2 = \xi_2(\xi, n, q) = (\xi - 1) + \xi^3 + \frac{\xi}{q} + \frac{3\xi^3}{n}.\]
Indeed, suppose that $0 \leq a < b < \infty$ and $\diam(\eta'([a,b])) \geq \dd \in (0,\ee^q]$. Set $k = \lceil \xi\log_2(1/\dd)\rceil$, noting that $K \leq k_0 \leq k$. Then we have
\[C_n'k2^{-k} < 2^{-k/\xi} \leq \dd \leq \diam (\eta'([a,b])).\]
Since $B_{n,\xi,k_0}$ holds, $\eta'([a,b])$ must leave $\D \sm B(0, 1 - 2^{-\xi k}/100)$. It follows that $\eta'([a,b])$ must come within distance $(3/4)2^{-m}$ of some point $z \in S_m$ as long as $2^{-m} \leq 2^{-\xi k}/400$, say, which holds if $m = \lceil \xi k + 9 \rceil$. In particular, $\eta'([a,b])$ will contain a subsegment $\eta'([c,d])$ with $\eta'([c,d]) \subseteq B(z, 2^{-m})$ and $\diam (\eta'([c,d])) \geq 2^{-m-2}$.

Since $H_{\ee, \xi, n}$ holds, $G_{\ee_{m,n}, z, 2^{-m}, \xi}$ holds where $\ee_{m,n} = 2^{-m-2-(3m/n)}\ee$, as before. Therefore, $\eta'([c,d])$ must contain a ball of radius $\rho = \ee_{m,n}^{\xi}$. Since $m \leq \xi k + 10$, we have
\begin{equation}\label{eq:area_1}
\rho = \ee_{m,n}^\xi = \left(2^{-m-2-(3m/n)}\ee\right)^\xi \geq 2^{-(12 + 30/n)\xi}\cdot 2^{-(1 + 3/n)\xi^2 k} \cdot \ee^\xi.
\end{equation}
Now, $k \leq \xi\log_2(1/\dd) +1$, so
\begin{equation}\label{eq:area_2}
2^{-(1 + 3/n)\xi^2 k} \geq \dd^{\xi^3(1 + 3/n)}\cdot 2^{-\xi^2(1+3/n)}.
\end{equation}

Since $n \geq 1$ and $\ee^{\xi - 1} < 2^{-42\xi-4\xi^2}$ (by assumption), we have $2^{-(12 + 30/n)\xi}\cdot 2^{-\xi^2(1+3/n)} > \ee^{\xi - 1} > \dd^{\xi - 1}$. Also, we have assumed $\dd \leq \ee^q$. Combining these with \eqref{eq:area_1} and \eqref{eq:area_2}, we have
\[\rho \geq \dd^{\xi - 1}\cdot \dd^{\xi^3(1 + 3/n)} \cdot \dd^{\xi/q} = \dd^{\xi - 1 + \xi^3 + 3\xi^3/n + \xi/q}.\]
That is, $E_{\ee_2, \xi_2}$ holds, as claimed.

By Steps 3 and 4, using $k_0 = \lceil \xi\log_2 (1/\ee)\rceil$, we have
\begin{equation}
    \p[H_{\ee, \xi, n}^c \cup B_{n,\xi, k_0}^c] \leq D_{n,\xi}\ee^n + 2\cdot 2^{-n k_0} \leq (D_{n,\xi} + 2)\ee^n. 
\end{equation}
Therefore it follows that for all $\ee \in (0, \ee_0(\xi))$, and for any choice of $q \in \N$, we have $\p[E_{\ee_2, \xi_2}] \leq (D_{n,\xi} + 2)\ee^n$.

\textit{Step 6. Conclusion.}
Fix $\xi > 1$ and $n \in \N$. Choose first $\xi_0 > 1$, then $n_0 \in \N$ and $q \in \N$ such that
\[\xi_2(\xi_0,q,n) = (\xi_0 - 1) + \xi_0^3 + \frac{\xi_0}{q} + \frac{3\xi_0^3}{n_0} < \xi, \qquad  \frac{n_0}{q} \geq n.\]
Then for $\ee \in (0, \ee_0(\xi_0)^{q})$ we have 
\[\p[E_{\ee,\xi}^c] \leq \p[E_{\ee_2(\ee^{1/q}, q),\xi_2(\xi_0, n_0, q)}^c] \leq (D_{n_0, \xi_0} + 2)\ee^{n_0/q} \leq (D_{n_0,\xi_0}+2)\ee^n.\]
It follows that for every choice of $\xi > 1$ and $n \in \N$, there exists a constant $A_{n,\xi} > 0$ such that for all $\ee > 0$, we have $\p[E_{\ee, \xi}^c] \leq A_{n,\xi}\ee^n$, completing the proof.

\end{proof}

\section{Proof of main theorem}\label{sec:proof}

We first define the conformal radius. If $D \subsetneq \C$ is a simply connected domain and $z \in D$, we define the \textit{conformal radius of $D$ viewed from $z$}, denoted by $\confrad(z,D)$, by $|f'(z)|\nv$, where $f$ is any conformal map from $D$ to $\D$ sending $z$ to $0$.

\begin{proof}[Proof of Theorem~\ref{thm:diam}.]
In the following, we fix $\kk$ either in $(8/3, 4)$ or in $(4,8)$, and we define $\kk' = 16/\kk$ if $\kk \in (8/3,4)$, and $\kk' = \kk$ if $\kk \in (4,8)$. This choice of notation will allow us to prove both cases at once using the same argument, and only one step will differ slightly depending on which case we are in. For $z \in \D\sm\Uu$, let $s_z$ denote the connected component of $\D \sm \Uu$ which contains $z$. Since $z \notin \Uu$ a.s.\ $s_z$ is defined for (Lebesgue) almost all $z \in \D$. We can write
\[\sum_{s_k \in \CS} \diam(s_k)^2 = \int_\D \frac{\diam(s_z)^2}{\Area(s_z)}d\mu(z),\]
where $\mu$ denotes Lebesgue measure on $\D$. 
Taking expectations we get
\begin{equation}\label{eq:sum_to_integral}
\E\left[\sum_{k = 1}^\infty \diam(s_k)^2\right] = \int_\D \E\left[\frac{\diam(s_z)^2}{\Area(s_z)}\right]d\mu(z).
\end{equation}
We will show that $\diam(s_z)^2/\Area(s_z)$ has finite expectation for each $z \in \D$, and that this expectation is integrable over $\D$, completing the proof.

Let $h$ be a GFF on $\D$ with boundary conditions given by $\la'$ to the left of $-i$ and $\la' - 2\chi\pi$ to its right, and let $\eta'$ be a space-filling $\SLE_\kk'(\kk'-6)$ loop in $D$ starting at $-i$ coupled with $h$ as described in Section~\ref{sec:sfsle}. Suppose that the $\CLE_\kk$ $\GG$ is coupled to $h$ as in Section~\ref{subsec:cle834} if $\kk \in (8/3, 4)$, or as in Section~\ref{subsec:cle48} if $\kk \in (4,8)$.
Fix $\xi \in (1,3/2)$, let $\ee \in (0,1)$, and assume the event $E_{\ee, \xi}$ holds, as defined in Proposition~\ref{prop:area_filling}.
Let $z \in \D$ and define $d_z = \diam(s_z)$. By Lemma~\ref{lem:sfsle_pocket_interaction_834} (for $\kk \in (8/3,4)$) and Lemma~\ref{lem:sfsle_pocket_interaction_48} (for $\kk \in (4,8)$), once $\eta'$ enters $s_z$, it a.s.\ hits every point of $s_z$ before leaving $\ov{s_z}$. 
In particular, on $E_{\ee,\xi}$, if $d_z = \dd$ for $\dd \in (0, \ee]$, the pocket $s_z$ must contain a ball of radius $\dd^\xi$.

Therefore, on the event $E_{\ee, \xi}$, if $d_z > \ee$ then $\Area(s_z) \geq \ee^{2\xi}$. If $d_z \leq \ee$ so that $d_z = \dd$ for some $\dd \in (0, \ee]$, then $\Area(s_z) \geq \dd^{2\xi} = d_z^{2\xi}$. In either case, on $E_{\ee,\xi}$, we have (noting that $d_z \leq 2$ always)
\begin{equation}\label{eq:bound_on_E}
\frac{\diam(s_z)^2}{\Area(s_z)} \leq \frac{4}{\ee^{2\xi}
} \vee d_z^{-2(\xi - 1)} \leq \frac{4}{\ee^{2\xi}
} + d_z^{-2(\xi - 1)}.
\end{equation}

For the sake of notation, redefine $E_{2, \xi} = \varnothing$ and suppress $\xi$ in the notation so that $E_\ee \equiv E_{\ee,\xi}$. 
Then we have $E_{2}\subset E_{1} \subset E_{1/2} \subset \cdots$, and by Proposition~\ref{prop:area_filling} the union $\cup_m E_{2^{-m}}$ has probability one.
Using \eqref{eq:bound_on_E} on each event $E_{2^{-m}}$, we have
\begin{align}
\begin{split}\label{eq:diam_area_sq_sum}
\E\left[\frac{\diam(s_z)^2}{\Area(s_z)}\right] &= \sum_{m=0}^\infty \E\left[\frac{\diam(s_z)^2}{\Area(s_z)} \one_{E_{2^{-m}}\sm E_{2^{-(m-1)}}}\right]\\
&\leq \sum_{m=0}^\infty \E\left[\left(\frac{4}{2^{-2m\xi}
} + d_z^{-2(\xi - 1)}\right)\one_{E_{2^{-m}}\sm E_{2^{-(m-1)}}}\right]\\
&= \E\left[d_z^{-2(\xi-1)}\right] + \sum_{m=0}^\infty \left(4\cdot 2^{2m\xi}\right)\p\left[E_{2^{-m}}\sm E_{2^{-(m-1)}}\right].
\end{split}
\end{align}
Intuitively, we can think of the first term as being the contributions coming from cases when the loop $s_z$ is small (for large loops, say $d_z > 1/2$, the quantity $d_z^{-2(\xi - 1)}$ is bounded, so won't pose any issue). Correspondingly, the sum comes from contributions when the loop is large relative to the length scale $2^{-m}$. Note that the second term does not depend on $z$, whereas the first term does.

By Proposition~\ref{prop:area_filling}, for any $n \in \N$ there exists $A = A_{n,\xi} > 0$ such that
\begin{equation}
\p[E_{2^{-m}}\sm E_{2^{-(m-1)}}] \leq \p[E_{2^{-(m-1)}}^c] \leq A2^{-(m-1)n}.
\end{equation}
Combining this with \eqref{eq:diam_area_sq_sum}, we obtain
\begin{equation}\label{eq:big_loop_sum}
\E\left[\frac{\diam(s_z)^2}{\Area(s_z)}\right] \leq \E\left[d_z^{-2(\xi-1)}\right] + (4A)2^n\sum_{m=0}^\infty 2^{2m\xi}2^{-mn}.
\end{equation}
For $n$ large enough ($n = 4$, say), this sum is finite and independent of $z$, so will be integrable over $\D$.

It remains to deal with $\E[d_z^{-2(\xi -1)}]$. Let $f_z \colon \D \to \D$ denote the unique conformal map sending $0$ to $z$ with $f'(0) > 0$. By the conformal invariance of $\CLE$, $s_z$ and $f_z(s)$ have the same law. We will first consider $d := d_0$, the diameter of $s := s_0$, the connected component of $\D\sm\Uu$ containing $0$. In \cite[Theorem~1]{ssw2009radii}, a probability density for $B_1^0 := -\log \confrad(0, s)$ is computed and implies \cite[Equation~(3)]{ssw2009radii} which states that for all $\la \in (0, \la_0(\kk))$ there exists a finite constant $C(\la, \kk)$ such that $\E[\exp(\la B_1^0)] = C(\la, \kk)$. Here, $\la_0(\kk) = 1 - 2/\kk - 3\kk/32$, which is positive exactly for $\kk \in (\frac83, 8)$. By \cite[Corollary~3.19]{lawler2008conformally}, we have $\frac14\confrad(0, s) \leq \dist(0, \del s) \leq 4\confrad(0,s)$ from which we can deduce $\frac14\confrad(0,s) \leq \diam(s) =: d$. It follows that $\E[d^{-\la}] < \infty$ for $\la \in (0, \la_0(\kk))$. Therefore, by choosing $\xi_0 > 1$ small enough that $2(\xi - 1) < \la_0(\kk)$, we can ensure that $\E[d^{-2(\xi -1)}] < \infty$ for all $\xi \in (1, \xi_0)$.

For general $z \in \D$, let $r = 1 - |z|$. As mentioned above we have $s_z \stackrel{d}{=} f_z(s)$. By \cite[Corollary~3.19 and Theorem~3.23]{lawler2008conformally}, for all $w \in \D$, 
\begin{equation}
|f_z(w) - f_z(0)| = |f_z(w) - z| \geq \frac{|w|}{4}|f_z'(0)| \geq \frac{|w|r}{16}.
\end{equation}
Since there must exist $w \in s$ with $\n{w} \geq d/4$, it follows that $|f_z(w) - z| \geq |w|r/16$, so $d_z \geq rd/64$ (where we couple $s_z$ and $s$ with different versions of the $\CLE_{\kk}$ so that $s_z = f_z(s)$). Therefore, there is a constant $D$, depending only on $\xi$, such that
\begin{equation}
\E[d_z^{-2(\xi -1)}] \leq D\, r^{-2(\xi -1)}\E[d^{-2(\xi -1)}].
\end{equation}
Therefore we see that
\begin{equation}\label{eq:integral_of_small_loop_bound}
\int_\D \E[d_z^{-2(\xi -1)}] d\mu(z) \leq D\E[d^{-2(\xi -1)}] \int_\D (1 - \n{z})^{-2(\xi - 1)}d\mu(z).
\end{equation}
Since $\E[d^{-2(\xi - 1)}] < \infty$ for $\xi \in (1, \xi_0)$, and since the integral of $(1-|z|)^{-\aal}$ over $\D$ is finite for $\aal < 1$, which is true when $\xi < 3/2$, it follows that by choosing $\xi > 1$ small enough, the right-hand side of \eqref{eq:integral_of_small_loop_bound} is finite. Finally, choosing $n \in \N$ large enough and $\xi > 1$ small enough, combining \eqref{eq:big_loop_sum} and \eqref{eq:integral_of_small_loop_bound} shows that \eqref{eq:sum_to_integral} is finite, completing the proof.
\end{proof}

\bibliographystyle{alpha}
\bibliography{references}

\end{document}